\documentclass[11pt]{article}

\usepackage{amsmath}
\usepackage{amsfonts}
\usepackage{color}
\usepackage{float}   
\usepackage{amssymb}
\usepackage{bm}    
\usepackage{graphicx}   
\usepackage{enumerate}
\usepackage{amsthm,amscd}
\usepackage[all]{xy}
\usepackage{multirow,makecell} 
\usepackage{pifont}        
 \usepackage{arydshln}    
 \usepackage{booktabs}  
 \usepackage{appendix}

\makeatletter
\def\hlinew#1{%
  \noalign{\ifnum0=`}\fi\hrule \@height #1 \futurelet
   \reserved@a\@xhline}
\makeatother       

 \allowdisplaybreaks

\usepackage{hyperref}

\newtheorem{theorem}{Theorem}[section]

\newtheorem{definition}[theorem]{Definition}
\newtheorem{conjecture}[theorem]{Conjecture}
\newtheorem{question}[theorem]{Question}

\newtheorem{lemma}[theorem]{Lemma}
\newtheorem{claim}{Claim}[section]

\begin{document} 

\title{A spectral extremal problem on 
non-bipartite\\ triangle-free graphs\thanks{This paper was firstly announced on April 3, 2023, and it was later published on Electron. J. Combin. 31 (1) (2024), \#P1.52. 
This is the final version; see \url{https://doi.org/10.37236/12009}.  
Following reviewer's suggestion, 
we have changed the original title 
``A solution of Zhai-Shu's question on spectral extremal problems''. 
E-mail addresses: \url{ytli0921@hnu.edu.cn} (Y. Li), 
\url{fenglh@163.com} (L. Feng),  
\url{ypeng1@hnu.edu.cn} (Y. Peng, corresponding author).}  }

\author{Yongtao Li$^{\dag}$, Lihua Feng$^{\dag}$, Yuejian Peng$^{\ddag}$ \\[2ex] 
{\small $^{\dag}$School of Mathematics and Statistics, Central South University} \\ 
{\small Changsha, Hunan, 410083, P.R. China}  
\\ 
{\small $^{\ddag}$School of Mathematics, Hunan University} \\
{\small Changsha, Hunan, 410082, P.R. China } 
 }


\maketitle

\vspace{-1cm}

\begin{abstract} 
A theorem of Nosal and Nikiforov states that 
if $G$ is a triangle-free graph with $m$ edges, then 
$\lambda (G)\le \sqrt{m}$, where the 
equality holds if and only if 
$G$ is a complete bipartite graph.  
A well-known spectral conjecture of Bollob\'{a}s and Nikiforov 
 [J. Combin. Theory Ser. B 97 (2007)] asserts that if $G$ is a $K_{r+1}$-free graph with $m$ edges, 
then $\lambda_1^2(G) + \lambda_2^2(G) \le (1-\frac{1}{r})2m$. 
Recently, Lin, Ning and Wu [Combin. Probab. Comput. 30 (2021)] 
confirmed the conjecture in the case $r=2$. 
Using this base case, they proved further that 
 $\lambda (G)\le \sqrt{m-1}$ for every non-bipartite triangle-free graph $G$, 
with equality  if and only if $m=5$ and $G=C_5$. 
Moreover, Zhai and Shu [Discrete Math. 345 (2022)] presented  an improvement 
by showing $\lambda (G) \le \beta (m)$, where $\beta(m)$ is the largest root of 
$Z(x):=x^3-x^2-(m-2)x+m-3$. 
The equality  in Zhai--Shu's result holds only if $m$ is odd and $G$ is obtained from the complete bipartite graph $K_{2,\frac{m-1}{2}}$ by subdividing 
exactly one edge. 
Motivated by this observation, Zhai and Shu proposed a question to 
find a sharp bound when $m$ is even. 
We shall solve this question by using a different method and 
 characterize three kinds of spectral extremal graphs 
over all triangle-free non-bipartite graphs with even size.  
Our proof technique is mainly based on applying Cauchy interlacing theorem of eigenvalues of a graph, and with the aid of a triangle counting lemma in terms of both eigenvalues and the size of a graph. 
 \end{abstract}



\section{Introduction}

 Let $G$ be a simple
 graph with vertex set $V(G)$ and edge set $E(G)$.  
 We usually write $n$ and $m$ for the number of vertices and edges,  
 respectively.  
One of the main problems of algebraic graph theory 
is to determine the combinatorial properties of a graph that are reflected from 
the  algebraic properties of  its associated matrices. 
Let $G$ be a simple graph on $n$ vertices. 
The \emph{adjacency matrix} of $G$ is defined as 
$A(G)=[a_{ij}]_{n \times n}$ where $a_{ij}=1$ if two vertices $v_i$ and $v_j$ are adjacent in $G$, and $a_{ij}=0$ otherwise.   
We say that $G$ has eigenvalues $\lambda_1 , \lambda_2,\ldots ,\lambda_n$ if these values are eigenvalues of 
the adjacency matrix $A(G)$. 
Let $\lambda (G)$ be the maximum  value in absolute 
 among all eigenvalues of  $G$, which is 
 known as the {\it spectral radius} of  $G$.

\subsection{The spectral extremal graph problems}

A graph $G$ is called {\it $F$-free} if it does not contain 
  an isomorphic copy of $F$ as a subgraph. 
Clearly, every bipartite graph is $C_{3}$-free. 
 The {\em Tur\'{a}n number} of a graph $F$ is the maximum number of edges  in an $n$-vertex $F$-free graph, and 
  it is usually  denoted by $\mathrm{ex}(n, F)$. 
  An $F$-free graph on $n$ vertices  with $\mathrm{ex}(n, F)$ edges is called an {\em extremal graph} for $F$. 
As is known to all, the Mantel theorem (see, e.g., \cite{Bollobas78}) asserts that 
if $G$ is a triangle-free graph on $n$ vertices, then 
\begin{equation}  \label{eq-man} 
 e(G) \le    \lfloor {n^2}/{4} \rfloor ,
 \end{equation}  
where the equality holds if and only if  $G$ 
is the balanced complete bipartite graph $K_{\lfloor \frac{n}{2}\rfloor, \lceil \frac{n}{2} \rceil }$.

  There are numerous extensions and generalizations 
  of Mantel's theorem; 
see  \cite{BT1981,Bon1983}.  Especially, 
  Tur\'{a}n  (see, e.g., \cite[pp. 294--301]{Bollobas78}) 
  extended Mantel's theorem by 
  showing that 
   if $G$ is a $K_{r+1}$-free graph on $n$ vertices with maximum number of edges, then $G$ is isomorphic to the graph $T_r(n)$, 
   where $T_r(n)$ denotes the complete $r$-partite graph whose part sizes are as equal as possible. 
Each vertex part of $T_r(n)$ 
has size either $\lfloor \frac{n}{r}\rfloor$ 
 or $\lceil \frac{n}{r}\rceil$. 
 The graph $T_r(n)$ is usually called Tur\'{a}n's graph. 
Five alternative proofs of Tur\'{a}n's theorem are selected into THE BOOK\footnote{Paul Erd\H{o}s liked to talk about THE BOOK, in which God maintains the perfect
proofs for mathematical theorems, and he also said that you
need not believe in God but you should believe in
THE BOOK.}  \cite[p. 285]{AZ2014}. 
Moreover, we refer the readers to the surveys  \cite{FS13, Sim13}.

Spectral extremal graph theory, with its connections and applications 
to numerous other fields, has enjoyed tremendous growth in the past few decades. 
There is a rich history on the study of bounding 
the  eigenvalues of a graph in terms of 
various parameters. For example, 
one can refer to \cite{BN2007jctb} for spectral radius and cliques, 
\cite{Niki2009jctb} for independence number and eigenvalues, 
\cite{TT2017,LN2021} for eigenvalues of outerplanar and planar graphs, 
\cite{CFTZ20, ZLX2022} for excluding friendship graph, 
and \cite{Tait2019, ZL2022jctb,HLF2023} for excluding minors.  
It is a traditional problem to bound the spectral radius  of a graph.   
Let $G$ be a graph on $n$ vertices with $m$ edges.   
It is natural to ask how large the spectral radius $\lambda (G)$ may be. 
A well-known result states that 
\begin{equation} \label{eqwell-known}
 \lambda (G)< \sqrt{2m}. 
  \end{equation}
This bound can be guaranteed by $\lambda(G)^2 < \sum_{i=1}^n \lambda_i^2 =\mathrm{Tr}(A^2(G)) =\sum_{i=1}^n d_i=2m$.  
We recommend the readers to \cite{Hong1988,HSF2001,Niki2002cpc} for more extensions.

It is also a popular problem 
to study the extremal structure for graphs with given number of edges. 
For example, it is not difficult to show that 
if $G$ has $m$ edges, then $G$ contains at most 
$\frac{\sqrt{2}}{3}m^{3/2}$ triangles; 
see, e.g., \cite[p. 304]{Bollobas78} and \cite{CC2021}. 
In addition, it is an instrumental topic to study 
the interplay between these two problems mentioned-above.  
More precisely, one can investigate 
 the largest eigenvalue of the adjacency matrix 
in a triangle-free graph with given number of edges\footnote{Note that when we consider the result on a graph with respect to 
the given number of edges, 
we shall ignore the possible isolated vertices if there are no confusions.}. 
Dating back to 1970,  Nosal \cite{Nosal1970} and Nikiforov \cite{Niki2002cpc,Niki2009jctb} independently obtained such a result.

 \begin{theorem}[Nosal \cite{Nosal1970}, Nikiforov \cite{Niki2002cpc,Niki2009jctb}] \label{thmnosal}
Let $G$ be a  graph  with $m$ edges. 
If $G$ is triangle-free, then 
\begin{equation}   \label{eq1}
\lambda (G)\le \sqrt{m} , 
\end{equation}
where the equality holds if and only if 
$G$ is a complete bipartite graph. 
\end{theorem}

Mantel's theorem in (\ref{eq-man}) 
can be derived from  (\ref{eq1}).  
Indeed, using Rayleigh's inequality, we have 
$\frac{2m}{n}\le \lambda (G)\le  \sqrt{m}$, 
which yields $ m \le \lfloor {n^2}/{4} \rfloor$.  
Thus, Theorem \ref{thmnosal} could be viewed as  
a spectral version of Mantel's theorem. 
Moreover, Theorem \ref{thmnosal} implies a result of Lov\'{a}sz and Pelik\'{a}n \cite{LP1973}, 
which asserts that if $G$ is a tree on $n$ vertices, then $\lambda (G)\le \sqrt{n-1}$, 
with equality  if and only if $G=K_{1,n-1}$.

Inequality (\ref{eq1}) 
impulsed the great interests of studying the maximum spectral radius 
for $F$-free graphs with given number of edges,  
see \cite{Niki2002cpc,Niki2009jctb} for  $K_{r+1}$-free graphs, 
\cite{Niki2009laa,ZS2022dm,Wang2022DM} for  $C_4$-free graphs,  
\cite{ZLS2021} for  $K_{2,r+1}$-free graphs,  
 \cite{ZLS2021,MLH2022} for $C_5$-free or $C_6$-free graphs, 
\cite{LLL2022} for $C_7$-free graphs, 
\cite{LSW2022,FYH2022,LW2022} for $C_4^{\triangle}$-free or $C_5^{\triangle}$-free graphs, where $C_k^{\triangle}$ 
is a graph on $k+1$ vertices obtained from $C_k$ 
and $C_3$ by sharing a common edge; 
see \cite{Niki2021} for  $B_k$-free graphs, 
where $B_k$ denotes the book graph consisting of $k$ triangles 
sharing a common edge, 
\cite{2022LLP} for $F_2$-free graphs with given number of edges, 
where $F_2$ 
is the friendship graph  consisting of two triangles intersecting in 
a common vertex, \cite{NZ2021,NZ2021b} for counting the number of $C_3$ and $C_4$. 
We refer the readers to the surveys \cite{NikifSurvey, LFL2022} 
and references therein.

\medskip 
 In particular, Bollob\'{a}s and Nikiforov \cite{BN2007jctb}  
posed the following nice conjecture.

\begin{conjecture}[Bollob\'{a}s--Nikiforov, 2007] \label{conj-BN}
Let $G$ be a $K_{r+1}$-free graph of order at least $r+1$ 
with $m$ edges. Then 
\begin{equation*}
{  \lambda_1^2(G)+ \lambda_2^2(G) 
\le 2m\Bigl( 1-\frac{1}{r}\Bigr) }. 
\end{equation*}
\end{conjecture}

Recently, Lin, Ning and Wu \cite{LNW2021} 
confirmed the base case $r=2$; see, e.g., \cite{Niki2021,LSY2022} for related results. 
Furthermore, the base case leads to 
Theorem \ref{thm-LNW} in next section.

\subsection{The non-bipartite triangle-free  graphs} 

The extremal graphs determined in 
Theorem \ref{thmnosal} are the complete bipartite graphs. 
Excepting the largest extremal graphs,
 the second largest extremal graphs were extensively studied 
 over the past years.  
 In this paper, we will pay attention mainly to the 
spectral extremal  problems 
for non-bipartite triangle-free graphs with given number of edges. 
Using the inequalities from majorization theory, 
Lin, Ning and Wu \cite{LNW2021} confirmed the triangle case
 in Conjecture \ref{conj-BN}, and then they proved the following result.

\begin{theorem}[Lin--Ning--Wu, 2021] \label{thm-LNW}
Let $G$ be a triangle-free graph with $m$ edges.  
If $G$  is non-bipartite, then 
\[  \lambda (G)\le \sqrt{m-1}, \] 
where the equality holds if and only if $m=5$ and $G=C_5$. 
\end{theorem}

The upper bound in Theorem \ref{thm-LNW} 
is not sharp for $m>5$. Motivated by this observation, 
 Zhai and Shu  \cite{ZS2022dm} provided a further improvement 
on Theorem \ref{thm-LNW}. 
For every integer $m\ge 3$, 
we denote by $\beta(m)$ the largest root of 
\begin{equation} \label{eq-Zx}
  Z(x):=x^3-x^2-(m-2)x+m-3. 
  \end{equation} 
If $m$ is odd, 
then we define $SK_{2,\frac{m-1}{2}} $ as the graph obtained 
from the complete bipartite graph $K_{2,\frac{m-1}{2}}$ by subdividing an edge; 
see Figure \ref{Fig-Zhai-Shu} for two drawings. 
Clearly, $SK_{2,\frac{m-1}{2}}$ is a triangle-free graph  with  
$m$ edges, and it is non-bipartite as it contains a copy of $C_5$.  
By computations, we know that  $\beta(m)$ is
 the spectral radius of  $SK_{2,\frac{m-1}{2}}$.

 \begin{figure}[H]
\centering 
\includegraphics[scale=0.8]{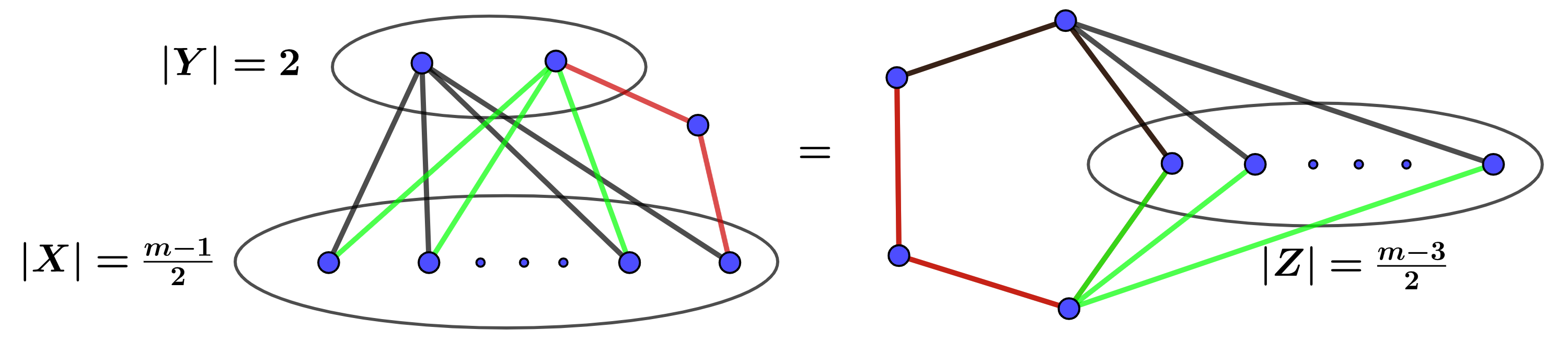}  
\caption{$m$ is odd and the graph $SK_{2,\frac{m-1}{2}} $. } \label{Fig-Zhai-Shu} 
\end{figure}

The improvement of Zhai and Shu \cite{ZS2022dm} on Theorem \ref{thm-LNW} 
can be stated as below. 

\begin{theorem}[Zhai--Shu, 2022] \label{thmZS2022}
Let $G$ be a  graph of size $m$. If $G$ is triangle-free and non-bipartite, then 
\[ \lambda (G) \le \beta (m), \]    
with equality  if and only if 
$G=SK_{2,\frac{m-1}{2}}$. 
\end{theorem}

Indeed, the result of Zhai and Shu improved Theorem \ref{thm-LNW}. 
It was proved in \cite[Lemma 2.2]{ZS2022dm} that   for every $m\ge 6$, 
 \begin{equation}  \label{eq-beta-ZS}
   \sqrt{m-2} < \beta(m) < \sqrt{m-1}. 
   \end{equation}   
The original proof of Zhai and Shu \cite{ZS2022dm} for Theorem \ref{thmZS2022} is technical and based on the use 
of the Perron components. Subsequently, Li and Peng \cite{LP2022oddcycle} 
provided an alternative proof by applying Cauchy interlacing 
theorem. We remark  that $\lim_{m\to \infty} 
(\beta(m)- \sqrt{m-2}) =0$.  
In addition, Wang \cite{Wang2022DM} improved Theorem \ref{thmZS2022} slightly 
by determining all the graphs with size $m$ 
whenever it is a non-bipartite triangle-free graph satisfying  $\lambda (G) \ge \sqrt{m-2}$.

\subsection{A question of Zhai and Shu}

The upper bound in Theorem \ref{thmZS2022} 
could be attained only if $m$ is odd, since the extremal graph $SK_{2,\frac{m-1}{2}}$ is well-defined only in this case. 
Thus, it is interesting to determine the spectral extremal graph 
when $m$ is even.  
Zhai and Shu in \cite[Question 2.1]{ZS2022dm} 
proposed the following question formally. 

\begin{question}[Zhai--Shu \cite{ZS2022dm}] \label{ques-ZS}
For even $m$, what is the extremal graph attaining the maximum 
spectral radius over all triangle-free non-bipartite graphs with $m$ 
edges? 
\end{question}

In this paper, we shall solve this question  
and determine the spectral extremal graphs. 
Although Question \ref{ques-ZS} seems to be another side 
of Theorem \ref{thmZS2022}, 
we would like to point out that the even case is actually more difficult and different, 
and the original method is ineffective in this case.

\begin{definition}[Spectral extremal graphs]  
Suppose that $m\in 2\mathbb{N}^*$. 
Let $L_m$ be the graph obtained from 
the subdivision $SK_{2,\frac{m-2}{2}}$ by  hanging an edge on a vertex with the maximum degree. 
If $\frac{m-3}{3}$ is a positive integer, 
then we define $Y_m$ as the graph obtained from $C_5$ by blowing up 
a vertex to an independent set $I_{\frac{m-3}{3}}$ 
on $\frac{m-3}{3}$ vertices, then adding a new vertex, and joining this vertex to all vertices of $I_{\frac{m-3}{3}}$. 
If $\frac{m-4}{3}$ is a positive integer, 
then we write $T_m$ for the graph obtained from $C_5$ 
by blowing up two adjacent vertices to independent sets $I_{\frac{m-4}{3}}$ 
and $I_2$, respectively, where $I_{\frac{m-4}{3}}$ and $I_2$ form a complete bipartite graph; 
see Figure \ref{fig-LFP}. 
\end{definition}

 \begin{figure}[H]
\centering 
\includegraphics[scale=0.8]{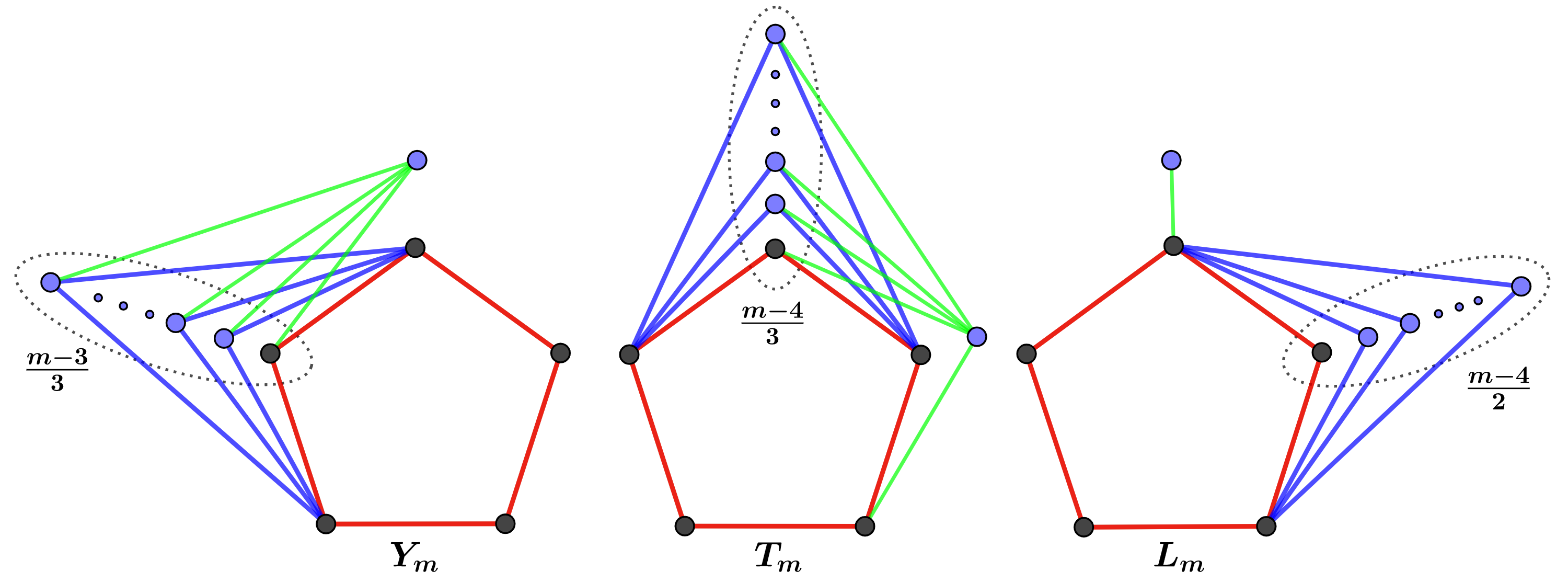}  
\caption{Extremal graphs in Theorem \ref{thm-main}. }
 \label{fig-LFP} 
\end{figure}

\begin{theorem}[Main result]  \label{thm-main} 
Let $m$ be even and $m\ge 4.7\times 10^5$. 
Suppose that $G$ is a triangle-free graph with $m$ edges 
and $G$ is non-bipartite.  \\ 
(a) If $m= 3t$ for  some $t\in \mathbb{N}^*$, then 
$ \lambda (G) \le \lambda (Y_m)$,   
with equality if and only if $G=Y_m$.  \\ 
(b) If $m= 3t +1$ for some $t\in \mathbb{N}^*$, then 
$  \lambda (G) \le \lambda (T_m)$,      
with equality  if and only if $G=T_m$.  \\  
(c) If $m= 3t +2$ for some $t\in \mathbb{N}^*$, then 
$ \lambda (G) \le \lambda (L_m)$,      
with equality if and only if $G=L_m$.  
\end{theorem}

The construction of $L_m$ is natural. 
Nevertheless, it is not apparent to find $Y_m$ and $T_m$. 
There are some analogous results 
 that the extremal graphs depend on the parity of the size $m$ in 
 the literature. 
For example, the $C_5$-free or $C_6$-free spectral extremal graphs 
with $m$ edges are determined in \cite{ZLS2021} when $m$ is odd, 
and later in \cite{MLH2022} when $m$ is even. 
Moreover, the $C_4^{\triangle}$-free or $C_5^{\triangle}$-free spectral extremal graphs are determined in \cite{LSW2022} for odd $m$, 
and subsequently in \cite{FYH2022,LW2022} for even $m$. 
In addition, the results of Nikiforov \cite{Niki2007laa2}, Zhai and Wang \cite{ZW2012} showed that the $C_4$-free spectral extremal graphs  with given order $n$ also rely on the parity of $n$. 
In a nutshell, for large size $m$, 
there is a common phenomenon that the extremal graphs in two cases 
are extremely similar, that is, 
the extremal graph in the even case  
is always constructed from that in the odd case by handing an edge to a 
vertex with maximum degree. 
Surprisingly, the extremal graphs in our conclusion break down this common phenomenon and show a new structure of the extremal graphs.

\medskip

\noindent 
{\bf Outline of the paper.} 
 In Section \ref{sec2}, we shall present 
 some lemmas, which  shows that 
the spectral radius of $L_m$ is smaller than that of $Y_m$ if $\frac{m}{3}\in \mathbb{N}^*$,  
as well as  that of  $T_m$ if $\frac{m-1}{3}\in \mathbb{N}^*$. 
Moreover, we will provide the estimations on 
both $\lambda (L_m)$ and $\beta (m)$.  
 In Section \ref{sec3}, 
 we will show some forbidden induced subgraphs, which helps us to characterize the local structure of the desired extremal graph. 
In Section \ref{sec4}, we present the proof of Theorem 
\ref{thm-main}. 
Our proof of Theorem \ref{thm-main} is quite different from 
that of Theorem \ref{thmZS2022}  in \cite{ZS2022dm}. 
The techniques used in our proof borrows some ideas from  
Lin, Ning and Wu \cite{LNW2021} as well as Ning and Zhai  \cite{NZ2021}. 
 We shall apply Cauchy's interlacing theorem 
 and a triangle counting result, which make full use of the information of all eigenvalues of a graph.  
In Section \ref{sec5}, we 
conclude this paper with some possible open  problems for interested readers.

\medskip 

\noindent 
{\bf Notations.} 
We shall follow the standard notation in  \cite{BM2008}  
and  consider only simple and undirected graphs. 
Let $N(v)$ be the set of neighbors of a vertex $v$, 
and $d(v)$  be the degree of $v$. 
For a subset $S\subseteq V(G)$, we write $e(S)$ for the 
number of edges with two endpoints in $S$, 
and $N_S(v)=N(v) \cap S$ for the set of neighbors of $v$ in $S$. 
Let  $K_{r+1}$ be the complete graph on $r+1$ vertices, 
and $K_{s,t}$ be the complete bipartite graph with parts of sizes 
  $s$ and $t$. 
   Let $I_k$ be an independent set on $k$ vertices.  
 We write  $C_n$ and $P_n$ for the cycle and 
 path on $n$ vertices, respectively. 
 Given graphs $G$ and $H$, we write $G\cup H$ for the 
 union of $G$ and $H$. In other words, $V(G\cup H)=V(G)\cup V(H)$ 
and $E(G\cup H)=E(G) \cup E(H)$. 
For simplicity, we write $kG$ for the union of $k$ copies of $G$. 
 We denote by $t(G)$ the number of triangles in $G$.

\section{Preliminaries and outline of the proof}

\label{sec2}

In this section, we will give the estimation on the spectral radius of $L_m$. 
Note that $L_m$ exists whenever $m$ is even,  
while  $Y_m$ and $T_m$ are  well-defined 
only if $m \,(\mathrm{mod}~3)$ is $0$ or $1$, respectively.  
We will show that $Y_m$ and $T_m$ 
have larger spectral radius than $L_m$. 
In addition, we will introduce Cauchy interlacing theorem,
 a triangle counting result in terms of eigenvalues, 
 and an operation of graphs which increases the spectral radius strictly. 
Before showing the proof of Theorem \ref{thm-main}, 
we will illustrate the key ideas of our proof, and 
then we outline the main steps of the framework.

\subsection{Bounds on the spectral radius of extremal graphs} 

By computations, 
we can obtain that $\lambda (Y_m)$ is the largest root of 
\begin{equation} \label{eq-Yx} 
 Y(x):=x^4 -x^3 +(2-m)x^2 + (m-3)x + \tfrac{m}{3}-1.  
  \end{equation}
Similarly, $\lambda (T_m)$ is the largest root of 
\begin{equation} \label{eq-Tx} 
 T(x):=x^5 - mx^3 + \tfrac{7m-22}{3} x + \tfrac{16-4m}{3},   
  \end{equation}
and $\lambda (L_m)$ is the largest root of 
the polynomial 
\begin{equation} \label{eq-Lx} 
 L(x) := x^6 - mx^4 + (\tfrac{5m}{2}-7)x^2 + (4-m)x + 2- \tfrac{m}{2}. 
 \end{equation}

 \begin{lemma} \label{lem-L-m}
 If $m\in \{6,8,10\}$, then $\lambda (L_m) > \sqrt{m-2}$. 
If $m\ge 12$ is even, then 
 \[   \sqrt{m-2.5} < \lambda (L_m) < \sqrt{m-2} . \]
 Moreover, we have $\lambda (L_6)\approx 2.1149$, 
 $\lambda (L_8) \approx 2.4938$ and $\lambda (L_{10}) \approx 2.8424$. 
 \end{lemma}

 \begin{proof} 
 The case $m\in \{6,8,10\}$ is straightforward. 
 Next, we shall consider the case $m\ge 12$. 
 By a direct computation,  it is easy to verify that 
 \[  L(\sqrt{m-2.5}) = -(1.25 + \sqrt{m-2.5}) m + 4\sqrt{m-2.5} +3.875 <0, \]
 which gives $\lambda (L_m) > \sqrt{m-2.5}$. 
 Moreover, we have  
 \[ L(\sqrt{m-2}) = \frac{1}{2} \left(m^2 -(9+2\sqrt{m-2})m + 8(2+\sqrt{m-2})  \right) >0.  \]
Furthermore, we have $  L'(x) := \frac{\mathrm{d}}{\mathrm{d} x} L(x)=6x^5 - 4mx^3 +(5m-14)x - m+4$. 
By calculations, one can check that $L'(\sqrt{m-2}) >0$ and $L'(x) \ge 0$ for every $x\ge \sqrt{m-2}$, 
which yields $L(x)> L(\sqrt{m-2}) >0$ for every $x> \sqrt{m-2}$. Thus 
$\lambda (L_m) < \sqrt{m-2}$. 
 \end{proof}

  \begin{lemma} \label{lem-LY}
 If $m\ge 38$ is even and $m=3t$ for some $t\in \mathbb{N}^*$, then 
 \[  \lambda (L_m) < \lambda (Y_m). \]
 \end{lemma}
 
 \begin{proof} 
 We know from (\ref{eq-Yx}) that $\lambda (Y_m)$ is the largest root of 
\[   Y(x) = x^4 -x^3 + (2-m) x^2 + (m-3) x + \tfrac{m-3}{3}. \] 
 By calculations, we can verify that
\[ L(x) -x^2Y(x) = 
x^5 -2x^4 + (3-m)x^3 + (\tfrac{13m}{6} - 6)x^2 + (4-m)x +2-\tfrac{m}{2}, \]
and for every $m\ge 38$, we have 
\[ L(x) -x^2Y(x) \Big|_{x=\sqrt{m-3}} =  
 \frac{m^2}{6} - m \sqrt{m-3} -m + 4\sqrt{m-3} +2>0.  \]
Moreover, we can show that $\frac{\mathrm{d}}{\mathrm{d} x} (L(x)- x^2Y(x)) >0$ 
for every $x\ge \sqrt{m-3}$. 
Thus, it follows that $L(x) > x^2Y(x)$ for every $x\ge \sqrt{m-3}$. 
So $ \lambda (L_m) < \lambda (Y_m)$, as needed. 
\end{proof}

  \begin{lemma} \label{lem-LT}
 If $m\ge 10$ is even and $m=3t+1$ for some $t\in \mathbb{N}^*$, then 
 \[  \lambda (L_m) < \lambda (T_m). \]  
 \end{lemma}
 
 \begin{proof} 
 Recall in (\ref{eq-Tx}) that 
 $\lambda(T_m)$ is the largest root of $T(x)$. 
It is sufficient to prove that $L(x)>xT(x) $ for every $x\ge 3$.  
Upon computation, we can get 
\[ L(x)- xT(x)= \frac{m+2}{6} x^2 + \frac{m-4}{3}x + \frac{4-m}{2}>0. \]
Consequently, we have $\lambda (L_m) < \lambda (T_m)$, as desired. 
 \end{proof}

The next lemma provides a refinement on (\ref{eq-beta-ZS}) for every $m\ge 62$. 

\begin{lemma} \label{lem-beta-m}
Let $m$ be even and $m\ge 62$. Then 
\[  \sqrt{m-2} < \beta (m) < \sqrt{m-1.85}. \]
\end{lemma}

\begin{proof}
Firstly, we have 
$Z(\sqrt{m-2})=-1 <0$, which yields $\sqrt{m-2} < \beta (m)$. 
Secondly, 
one can check that $Z(\sqrt{m-1.85}) >0$ for every $m\ge 62$, 
and $Z'(x)=3x^2 -2x -(m-2) >0$ for $x\ge \sqrt{m-1.85}$. 
Therefore, we have $Z(x)>Z(\sqrt{m-1.85}) >0$ for every $x> \sqrt{m-1.85}$,  
which yields $\beta (m) < \sqrt{m-1.85}$, as required.  
\end{proof}

The following lemma is referred to as the eigenvalue interlacing 
theorem, also known as Cauchy interlacing theorem, 
which states that the eigenvalues of a principal submatrix of a 
Hermitian matrix interlace those of the underlying matrix; 
see, e.g., \cite[pp. 52--53]{Zhan13} or \cite[pp. 269--271]{Zhang11}. 
The eigenvalue interlacing theorem 
is a powerful tool to extremal combinatorics and 
 plays a significant role in two recent breakthroughs \cite{Huang2019,JTYZZ2021}.

\begin{lemma}[Eigenvalue Interlacing Theorem] \label{lemCauchy}
Let $H$ be an $n \times n$ Hermitian matrix partitioned as 
\[ H= \begin{bmatrix} A & B \\ B^* & C  \end{bmatrix} , \]
where $A$ is an $m\times m$ principal submatrix of $H$ for some 
$m\le n$. Then for every $1\le i\le m$, 
\[  \lambda_{n-m+i}(H) \le \lambda_i (A) \le \lambda_i(H). \]
\end{lemma}

Recall that $t(G)$ denotes the number of triangles in $G$.  
It is well-known that the value of $(i,j)$-entry of $A^k(G)$ 
is equal to the number of walks of length $k$ in $G$ 
starting from vertex $v_i$ 
to $v_j$. 
Since each triangle of $G$ contributes $6$ closed walks of length $3$, 
we can count the number of triangles and obtain 
\begin{equation} \label{count}
t(G)=\frac{1}{6} \sum_{i=1}^n A^3(i,i)= \frac{1}{6}\mathrm{Tr}(A^3)=\frac{1}{6}\sum\limits_{i=1}^n\lambda_i^3.
 \end{equation}

The forthcoming lemma 
could be regarded as a triangle spectral 
counting lemma  in terms of  
both the eigenvalues and the size of a graph. 
This could be viewed as a useful variant of (\ref{count}) by using $\sum_{i=1}^n \lambda_i^2 = \mathrm{tr}(A^2)=
\sum_{i=1}^n d_i=2m$. 

\begin{lemma}[see \cite{NZ2021}] \label{lem21} 
Let $G$ be a graph on $n$ vertices with $m$ edges. 
If $\lambda_1\ge \lambda_2 \ge \cdots \ge \lambda_n$ 
are all eigenvalues of $G$, then  
\[ t(G)=
\frac{1}{6} \sum_{i=2}^n (\lambda_1 + \lambda_i) \lambda_i^2 + 
\frac{1}{3}(\lambda_1^2-m)\lambda_1.  \] 
\end{lemma}

For convenience, we introduce a function $f(x)$, 
which will be frequently used in Section \ref{sec3} 
to find the induced substructures that are forbidden in the extremal graph. 

\begin{lemma} \label{lem-fx}
Let $f(x)$ be a function given as 
\[  f(x) :=(\sqrt{m-2.5} + x)x^2. \] 
 If $a\le x \le b \le 0$, 
then 
\[  f(x)\ge \min\{ f(a),f(b) \}. \]  
\end{lemma}

\begin{proof}
The function $f(x)$ is  increasing when 
$x\in (-\infty, -\frac{2}{3}\sqrt{m-2.5})$, 
and  decreasing when $x\in [-\frac{2}{3}\sqrt{m-2.5}, 0]$. 
Thus the desired statement holds immediately. 
\end{proof}

\medskip
The following lemma \cite{WXH2005} is also needed in this paper,
it provides an operation on a connected graph and  increases the adjacency spectral radius strictly.

\begin{lemma}[Wu--Xiao--Hong \cite{WXH2005}, 2005] \label{lem-WXH}
Let $G$ be a connected graph
and $(x_1,\ldots ,x_n)^T$ be a Perron vector of $G$,
where $x_i$ corresponds to $v_i$.
Assume that
 $v_i,v_j \in V(G)$ are vertices such that $x_i \ge x_j$, and $S\subseteq N_G(v_j) \setminus N_G(v_i)$ is  a non-empty set.
 Denote $G^*=G- \{v_jv : v\in S\} +
\{v_iv : v\in S\}$. Then $\lambda (G) < \lambda (G^*)$.
\end{lemma}

\subsection{Proof overview}

As promised, we will interpret the key ideas and steps of the proof of 
Theorem \ref{thm-main}. 
First of all, we would like to make a comparison of 
the proofs of Theorem \ref{thm-LNW} and Theorem \ref{thmZS2022}. 
The proof of Theorem \ref{thm-LNW} in \cite{LNW2021}  is short and succinct.  
It relies on the base case in 
Conjecture \ref{conj-BN}, which states that 
if $G$ is a triangle-free graph with $m\ge 2$ edges, then 
\begin{equation}  \label{eq-BN-tri} 
\lambda_1^2(G) + \lambda_2^2(G) \le m, 
\end{equation}
where the equality holds if and only if $G$ is one of some specific bipartite graphs; see \cite{LNW2021,Niki2021}. 
Combining the condition in Theorem \ref{thm-LNW}, 
we know that if $G$ is a triangle-free non-bipartite graph such that  $\lambda_1 (G) \ge \sqrt{m-1}$, 
then $\lambda_2(G)<1$. 
Such a bound on the second largest eigenvalue 
provides great convenience  to characterize the local structure 
of $G$. For instance,  
combining $\lambda_2(G)<1$ with 
the Cauchy interlacing theorem, we obtain that $C_5$ is 
a shortest odd cycle 
of $G$. 
However, it is not sufficient to use  
(\ref{eq-BN-tri})  for the proof of Theorem \ref{thmZS2022}.   
Indeed, if  $G$ satisfies further that $\lambda (G) \ge 
\beta (m)$, 
then we get $\lambda_2(G) <2$ only, 
since $\beta(m) \to  \sqrt{m-2}$ as $m$ tends to infinity. 
Nevertheless, 
this bound is invalid for our purpose to describe the local structure of $G$. 
The original proof of 
Zhai and Shu \cite{ZS2022dm} for Theorem \ref{thmZS2022}  
 avoids the use of  (\ref{eq-BN-tri}) and 
 applies the Perron components. Thus it needs to
make more careful structure analysis of the desired extremal graph. 
 
\medskip  
To overcome the aforementioned obstacle, 
we will get rid of the use of (\ref{eq-BN-tri}),
and then exploit the information of all eigenvalues 
of graphs, instead of the second largest eigenvalue merely.    
Our proof of Theorem \ref{thm-main} grows out  from  
the original proof \cite{LNW2021} of Theorem \ref{thm-LNW}, 
which provided a method to find forbidden induced substructures. 
We will frequently use Cauchy interlacing theorem and 
the triangle counting result in Lemma \ref{lem21}.

\medskip 
The main steps of our proof can be outlined as below.  
It introduces the main ideas of the approach of this paper 
for treating the problem involving triangles. 

\begin{itemize}
\item[$\spadesuit$]
Assume that $G$ is a spectral extremal graph with even size, that is, $G$ is a non-bipartite triangle-free graph 
and attains the maximum spectral radius. 
First of all, we will show that $G$ is connected and 
it does not contain the odd cycle $C_{2k+1}$ as an induced subgraph 
for every $k\ge 3$. Consequently, 
$C_5$ is a shortest odd cycle in $G$.

\item[$\heartsuit$] 
Let $S$ be the set of vertices of a copy of $C_5$ in $G$.  
By using Lemma \ref{lemCauchy} and 
Lemma \ref{lem21},  we will find more 
forbidden substructures in the desired extremal graph; 
see, e.g., the graphs $H_1,H_2,H_3$ in Lemma \ref{lem-H123}. 
In this step, we will characterize and refine 
the local structure on the vertices around the cycle $S$.

\item[$\clubsuit$] 
Using the information on the local structure of $G$, 
we will show that $V(G)\setminus S$ has at most one vertex 
with distance two to $S$; see Claim \ref{claim-42}. 
Moreover, there are at most three vertices of $V(G)\setminus S$ with exactly one neighbor on  $S$, 
and all these vertices are adjacent to a same vertex of $S$. 

\item[$\diamondsuit$] 
Combining with the three steps above, 
we will determine the structure of $G$ and 
show some possible graphs with large spectral radius. 
By comparing the polynomials of graphs, 
we will prove that $G$ is isomorphic to $Y_m,T_m$ or $L_m$. 
\end{itemize}

\section{Some forbidden induced subgraphs} 

\label{sec3}

In this section, we always assume that $G$ is a 
non-bipartite triangle-free graph with even size $m$ and 
$G$ attains the maximal spectral radius. Since $L_m$ 
 is triangle-free and non-bipartite, 
 we get by Lemma \ref{lem-L-m} that 
\begin{equation}  \label{eq-G-25}
\lambda (G)\ge \lambda (L_m) 
> \sqrt{m-2.5}. 
\end{equation}
On the other hand, 
we obtain from  Theorem \ref{thmZS2022} 
and Lemma \ref{lem-beta-m} that 
\begin{equation} \label{eq-1.85}
\lambda (G)< \beta(m) < \sqrt{m-1.85}. 
\end{equation} 

Our aim in this section is to determine some
 forbidden induced substructures of the extremal graph $G$. 
 In this process, we need to exclude $16$ induced substructures for our purpose.  
 One of the main research directions in the proof 
 is to show that $G$ has at least one triangle, i.e., $t(G)>0$, whenever the substructure forms an induced 
 copy in $G$. 
 Throwing away some tedious calculations, 
 the main tools used in our proof attribute to 
 Cauchy Interlacing Theorem (Lemma \ref{lemCauchy}) 
and the triangle counting result (Lemma \ref{lem21}).

\begin{lemma} \label{lem-C5}
For any odd integer $s\ge 7$, 
an extremal graph $G$ does not contain $C_s$ 
as an induced cycle. 
 Consequently,  $C_5$ is a shortest odd cycle in $G$. 
\end{lemma}

\begin{proof}
Since $G$ is non-bipartite, let $s$ be the length of a shortest odd cycle 
in $G$. Since $G$ is triangle-free, we have $s\ge 5$. 
Moreover, a shortest odd cycle $C_s \subseteq G$ 
must be an induced odd cycle. It is well-known that the eigenvalues of $C_s$ are given as 
$ \left\{2\cos \frac{2\pi k}{s}: k=0,1,\ldots ,s-1  \right\}$.  
In particular, we have 
\[ \mathrm{Eigenvalues}(C_7)= \{2,1.246, 1.246,-0.445, -0.445, -1.801, -1.801\}. \]
Since $C_s$ is an induced copy in $G$, 
we know that $A(C_s)$ is a principal submatrix of $A(G)$. 
Lemma \ref{lemCauchy} implies that for every $i\in \{1,2,\ldots ,s\}$, 
\[ \lambda_{n-s+i}(G) \le  \lambda_i(C_s) \le \lambda_i (G). \]  
where $\lambda_i$ means the $i$-th largest eigenvalue. 
We next show that $s=5$. 
For convenience, we write $\lambda_1\ge \lambda_2 \ge \cdots \ge \lambda_n$ 
for  eigenvalues of $G$ in the non-increasing order.

Suppose on the contrary that $C_7$ is an induced odd cycle of $G$,
 then $\lambda_2 \ge \lambda_2(C_7) = 2\cos \frac{2\pi}{7}\approx 1.246$  
 and $\lambda_3 \ge \lambda_3(C_7)=2 \cos \frac{12\pi}{7} \approx 1.246$.  
 Recall in Lemma \ref{lem-fx} that 
\[  f(x) =(\sqrt{m-2.5} + x)x^2. \] 
Evidently, we get 
 \[  f(\lambda_2) \ge f(1.246) 
 \ge 1.552 \sqrt{m-2.5} + 1.934 \]
 and 
  \[  f(\lambda_3) \ge f(1.246)  \ge 1.552 \sqrt{m-2.5} + 1.934. \]
  Our goal is to get a contradiction by applying Lemma \ref{lem21} and showing $t(G)>0$. 
  It is not sufficient to obtain $t(G)>0$ by using the  positive eigenvalues of $C_7$ only. 
  Next, we are going to exploit the negative eigenvalues of $C_7$. 
For $i\in \{4,5,6,7\}$, we know that $\lambda_i(C_7)<0$. 
The Cauchy interlacing theorem yields 
$\lambda_{n-3} \le \lambda_4(C_7)=-0.445$,
$\lambda_{n-2} \le \lambda_5(C_7)=-0.445$, 
$\lambda_{n-1} \le \lambda_6(C_7)=-1.801$ 
and $\lambda_{n} \le \lambda_7(C_7)=-1.801$. 
To apply Lemma \ref{lem-fx}, we need to 
find the lower bounds on $\lambda_i $ for each 
$i\in \{n-3,n-2,n-1,n\}$. We know from (\ref{eq-G-25}) that 
$\lambda_1 \ge \lambda (L_m) > \sqrt{m-2.5}$, and  then 
$\lambda_n^2 \le 2m - (\lambda_1^2 +\lambda_2^2 + 
\lambda_3^2 + \lambda_{n-3}^2 + \lambda_{n-2}^2 + 
\lambda_{n-1}^2 ) < 2m-(m-2.5 + 6.744) = m- 4.244$, 
which implies $-\sqrt{m-4.244} < \lambda_n \le -1.801$. By Lemma \ref{lem-fx}, we get 
\[ f(\lambda_n) \ge 
\min\{f(-\sqrt{m-4.244}), f(-1.801)\}> 0.8\sqrt{m-2.5} . \]
 Similarly, 
we have $\lambda_{n-1}^2+\lambda_n^2 \le 
2m - (\lambda_1^2 +\lambda_2^2 + 
\lambda_3^2 + \lambda_{n-3}^2 + \lambda_{n-2}^2) 
< m- 1.001$. Combining with $\lambda_{n-1}^2 \le \lambda_n^2$, 
we get $-\sqrt{(m-1.001)/2} < \lambda_{n-1} \le -1.801$. 
By Lemma \ref{lem-fx}, we obtain  
\[  f(\lambda_{n-1}) \ge 
\min \{f(-\sqrt{(m-1.001)/2}), f(-1.801)\} > 
3.243\sqrt{m-2.5} - 5.841.  \]
Using (\ref{eq-G-25}) and (\ref{eq-1.85}), 
we have $  \sqrt{m-2.5} < \lambda_1 < \sqrt{m-1.85}$.  
By Lemma \ref{lem21}, we get 
\begin{align*} 
t(G)  &> \frac{1}{6} (f(\lambda_2) + f(\lambda_3) + f(\lambda_n) 
+ f(\lambda_{n-1}) )
 - \frac{2.5}{3}\lambda_1 \\
 & >\frac{1}{6} (7.147 \sqrt{m-2.5} - 5\sqrt{m-1.85}- 1.973) >0 . 
\end{align*} 
This is a contradiction. 
By the monotonicity of $\cos x$, we can prove that 
$C_s$ can not be an induced subgraph of $G$ for each odd integer $s\ge 7$. 
Thus we get $s=5$. 
\end{proof}

Using a similar method as in the proof of Lemma \ref{lem-C5}, 
we can prove the following lemmas, whose proofs are postponed to the Appendix. To avoid unnecessary calculations, we did not attempt to get the best bound on the size of $G$, and then we consider the case $m\ge 4.7\times 10^5$.

\begin{lemma} \label{lem-H123}
 $G$ does not contain any graph of 
 $\{H_1,H_2,H_3\}$ as an induced 
 subgraph. 
 \end{lemma}

 \begin{figure}[H]
\centering 
\includegraphics[scale=0.8]{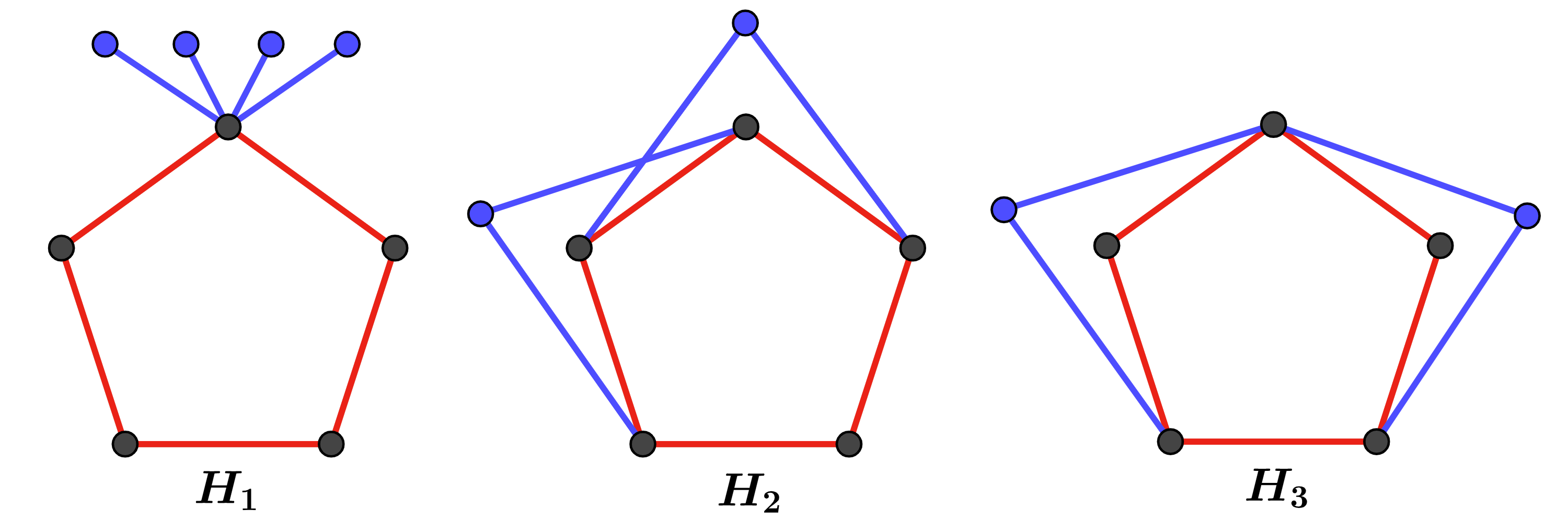} 
\label{fig-H123} 
\end{figure}

  \begin{lemma} \label{lem-T1234}
 $G$ does not contain any graph of 
 $\{T_1,T_2,T_3,T_4\}$ as an induced 
 subgraph. 
 \end{lemma}

 \begin{figure}[H]
\centering 
\includegraphics[scale=0.8]{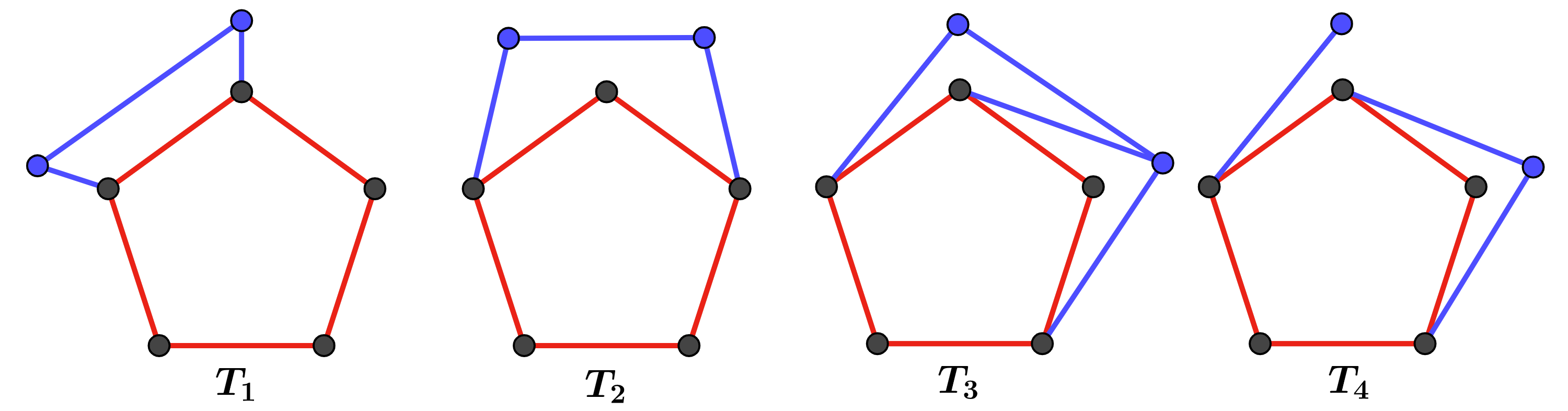} 
\label{fig-T1234} 
\end{figure}

  \begin{lemma} \label{lem-J1234}
 Any graph of 
 $\{J_1,J_2,J_3,J_4\}$ can not be an induced 
 subgraph of $G$. 
 \end{lemma}
 
 \begin{figure}[H]
\centering 
\includegraphics[scale=0.8]{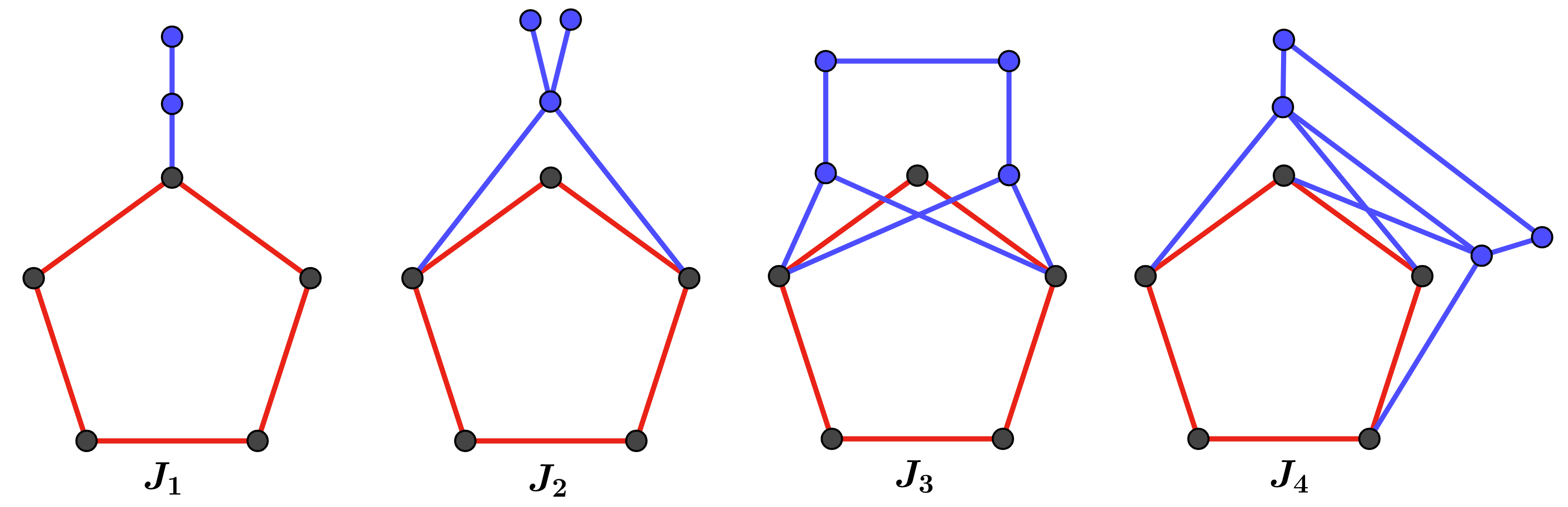} 
\label{fig-J1234} 
\end{figure}

  \begin{lemma} \label{lem-L1234}
Any graph of 
 $\{L_1,L_2,L_3,L_4\}$ can not be  an induced 
 subgraph of $G$. 
 \end{lemma}

 \begin{figure}[H]
\centering 
\includegraphics[scale=0.8]{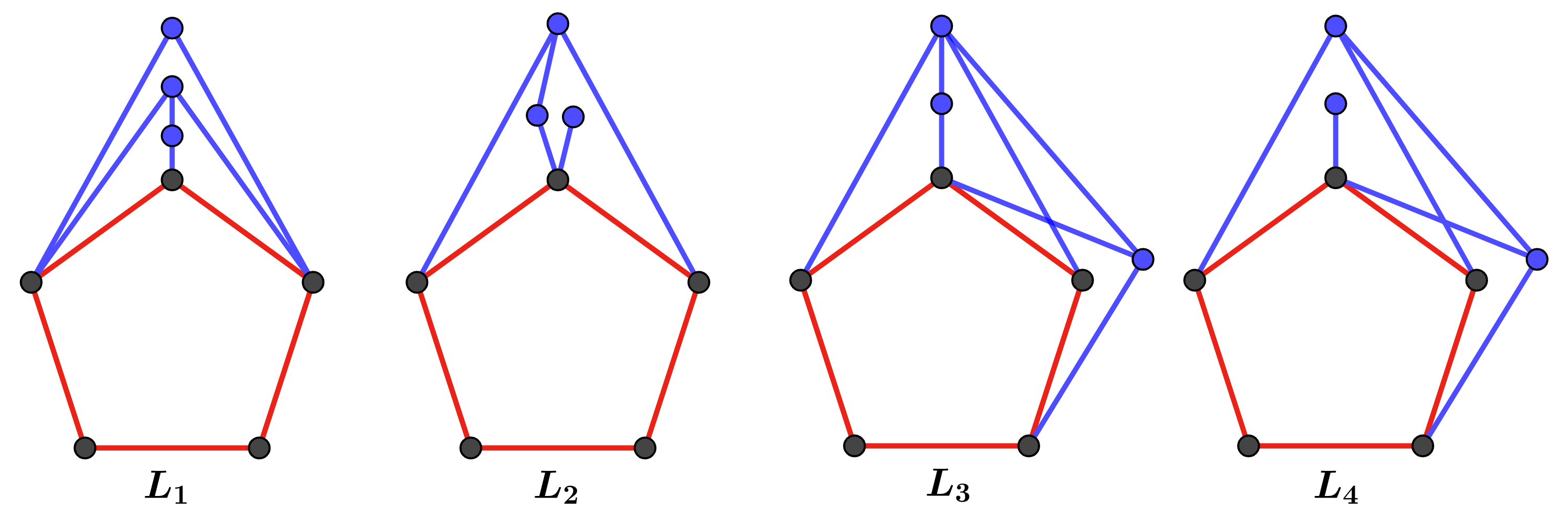} 
\label{fig-L1234} 
\end{figure}

\section{Proof of the main theorem} 

\label{sec4}

It is the time to show the proof of Theorem \ref{thm-main}.

\begin{proof}[{\bf Proof of Theorem \ref{thm-main}}]
Suppose that $G$ is a non-bipartite triangle-free graph 
with $m$ edges ($m\ge 4.7\times 10^5$ is even) 
 such that $G$ attains the maximum spectral radius. 
 Thus we have $\lambda (G)\ge \lambda (L_m)$ since $L_m$ 
 is one of the triangle-free  non-bipartite graphs. 
Our goal is to prove that  $G=Y_m$ if $\frac{m}{3} \in \mathbb{N}^*$; 
$G=T_m$ if $\frac{m-1}{3} \in \mathbb{N}^*$,     
and $G=L_m$ if $\frac{m-2}{3} \in \mathbb{N}^*$.  
First of all, we can see that 
$G$ must be connected.  
Otherwise, we can choose $G_1$ and $G_2$ as two different components, 
where $G_1$ attains the spectral radius of $G$.  
By identifying two vertices from $G_1$ and $G_2$,  respectively, 
we get a new graph with larger  spectral radius, 
which is a contradiction. 
By Lemma \ref{lem-C5}, 
we can draw the following claim.

\begin{claim} \label{claim-C5}
 $C_5$ is a shortest odd cycle in $G$. 
\end{claim}

By Claim \ref{claim-C5}, we denote by  $S=\{u_1,u_2,u_3,u_4,u_5\}$ the set of vertices of a copy of $C_5$, 
where $u_iu_{i+1}\in E(G)$ 
and $u_5u_1\in E(G)$.   Let $N(S):=
\bigl(\cup_{u\in S} N(u) \bigr)\setminus S$ be the union of 
 neighborhoods of vertices of $S$, 
 and let $d_S(v)=|N(v)\cap S|$ be the 
 number of neighbors of $v$ in the set $S$.  
   Clearly, 
   we have $d_S(v)\in \{0,1,2\}$ for every $v\in V(G) \setminus S$. 
   Otherwise, if $d_S(v) \ge 3$, then one can find a triangle immediately, a contradiction.

\begin{claim} \label{claim-42}
$V(G) \setminus S$ does not contain a vertex with distance $3$ to $S$, 
and $V(G)\setminus S$ has at most one vertex with distance $2$ to $S$. 
\end{claim}

\begin{proof} 
This claim is a consequence of Lemmas \ref{lem-H123} and 
\ref{lem-J1234}. 
Firstly, suppose on the contrary that $V(G)\setminus S$ contains 
a vertex which has distance $3$ to $S$. 
Let $w_1$ be such a vertex and 
$P_4=w_1w_2w_3u_1$ be a shortest path of length $3$. Then 
$w_2$ can not be adjacent to any vertex of $S$. 
Since $G$ is triangle-free, we know that  
neither $w_3u_2$ nor $w_3u_5$ can be an edge,
and at least one of $w_3u_3$ and $w_3u_4$ is not an edge. 
If $w_3u_3\notin E(G)$ and $w_3u_4\notin E(G)$, 
then $\{w_2,w_3\}\cup S$ 
induces a copy of $J_1$, contradicting with Lemma \ref{lem-J1234}. 
If $w_3u_3\in E(G)$ and $w_3u_4 \notin E(G)$, then  $\{w_1,w_2,w_3\}\cup 
(S \setminus \{u_2\})$ 
forms an induced copy of $J_1$ since $w_1w_3,w_1u_i $ and $w_2u_i$ are not edges of $G$,  a contradiction. 
By symmetry, $w_3u_3\notin E(G)$ and $w_3u_4\in E(G)$ 
 yield a contradiction similarly. 

Secondly, suppose on the contrary that 
$V(G)\setminus S$ contains two vertices, say $w_1,w_2$, 
which have distance $2$ to $S$. Let $v_1$ and $v_2$ be two vertices out of $S$ such that 
 $w_1\sim v_1 \sim S$ 
and $w_2\sim v_2 \sim S$. 
Since $J_1$ can not be an induced copy of $G$ and $G$ is triangle-free, 
we know that $d_S(v_1)=d_S(v_2)=2$.  
If $v_1=v_2$, then $\{w_1,w_2,v_1\} \cup S$ forms an induced copy 
of $J_2$ in $G$, we get a contradiction by Lemma \ref{lem-J1234}.  
Thus, we get $v_1\neq v_2$. 
Without loss of generality, 
we may assume that $N_S(v_1)=\{u_1,u_3\}$. 
By Lemma \ref{lem-H123}, $G$ does not contain $H_3$ as an induced subgraph, 
we get $N_S(v_2)\neq  \{u_3,u_5\}$ 
and $N_S(v_2) \neq \{u_1,u_4\}$. 
By symmetry, we have either $N_S(v_2) 
= \{u_2,u_4\}$ or $N_S(v_2)= \{u_1,u_3\}$. 
For the former case,  
since $H_2$ is not an induced subgraph of $G$ by Lemma \ref{lem-H123}, 
we get $v_1v_2 \in E(G)$. 
If $w_1w_2 \in E(G)$, then $G$ contains $J_4$ as an induced subgraph, 
which is a contradiction by Lemma \ref{lem-J1234}. 
Thus $w_1w_2 \notin E(G)$. 
By Lemma \ref{lem-WXH},  
one can compare the Perron components of $v_1$ and $v_2$, 
and then move $w_1$ and $w_2$ together, namely, 
either making $w_1$ adjacent to $v_2$, or  $w_2$ adjacent to $v_1$. 
In this process, the resulting graph remains triangle-free and non-bipartite 
as well. 
However, it has larger spectral radius than $G$, 
which contradicts with the maximality of the spectral radius of $G$. 
For the latter case, i.e., $N_S(v_1)=N_S(v_2)=\{u_1,u_3\}$. 
Since $J_3$ is not an induced copy in $G$,  
a similar argument shows $w_1w_2\notin E(G)$, and then it also leads to a contradiction. 
\end{proof}

By Claim \ref{claim-42}, we shall partition the remaining proof in two cases,  
which are dependent on whether $V(G)\setminus S$ contains a vertex with distance $2$ 
to the $5$-cycle $S$.

\medskip 
 {\bf Case 1.} Every vertex of $V(G)\setminus S$ is adjacent a vertex of $S$.

 In this case, we have $V(G)=S \cup N(S)$. For convenience, 
 we denote $N(S) = V_1\cup V_2$, 
 where $V_i=\{v\in N(S): d_S(v)=i\}$ for each $i=1,2$.   
 At the first glance, different vertices of $V_1$ can be joined to 
 different vertices of $S$. 
By Lemma \ref{lem-T1234}, $G$ does not contain $T_1$ and $T_2$ as induced subgraphs, 
 we obtain that $V_1$ is an independent set in $G$. 
Using Lemma \ref{lem-WXH}, we can move all vertices of $V_1$ together such that 
 {\it all of them are adjacent to a same vertex of $S$}, and  get a new graph 
 with larger spectral radius. 
Note that this process can keep the resulting graph being 
triangle-free and non-bipartite since $V_1$ is edge-less and 
 $S$ is still a copy of $C_5$. 
By Lemma \ref{lem-H123}, $H_1$ can not be an induced subgraph of $G$, 
then $|V_1|\le 3$.

We can fix a vertex $v\in N(S)$ and 
assume that $N_S(v)=\{u_1,u_3\}$. 
For each $w\in V(G)\setminus (S\cup \{v\})$, 
since $G$ contains no triangles and  no  
$H_3$ as an induced subgraph by Lemma \ref{lem-H123}, 
we know that 
$N_S(w)\neq \{u_3,u_5\}$ and $N_S(w)\neq \{u_4,u_1\}$. 
It is possible that $N_S(w)=\{u_1,u_3\},\{u_2,u_4\}$ or 
$\{u_5,u_2\}$. Furthermore, 
if $N_S(w)=\{u_1,u_3\}$, then $wv \notin E(G)$ since $G$  contains no  triangle; 
if $N_S(w)=\{u_2,u_4\}$, then 
$wv \in E(G)$ since $G$ contains no induced copy of 
$H_2$. 
We denote $N_{i,j} :=\{w\in V(G)\setminus S : N_S(w)=\{u_i,u_j\}\}$. Note that $G$ has no induced copy of $H_3$, 
 then at least one of the sets $N_{2,4}$ and $N_{5,2}$ is empty. 

{\bf Subcase 1.1.} 
 If both $N_{2,4}=\varnothing$ and $N_{5,2}=\varnothing$, 
then $V_2=N_{1,3}$ and $V(G) = S\cup V_1\cup N_{1,3}$. 
By Lemma \ref{lem-T1234}, $T_3$ and $T_4$ can not be induced subgraphs of $G$. 
Hence, all vertices of $V_1$ are adjacent to the vertex $u_1$ or 
$u_2$ by symmetry. 
Next, we will show  that $|V_1|\in \{1,3\}$, and then we prove that 
$V_1$ and $N_{1,3}$ form  a complete bipartite graph or an empty graph. 
  
 \begin{figure}[H]
\centering 
\includegraphics[scale=0.8]{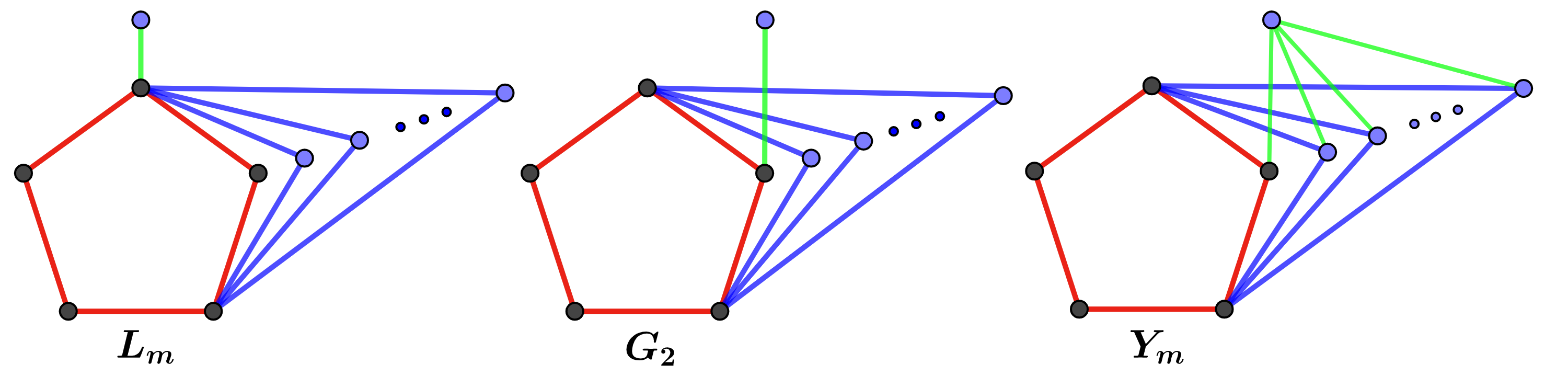} \\ 
\includegraphics[scale=0.8]{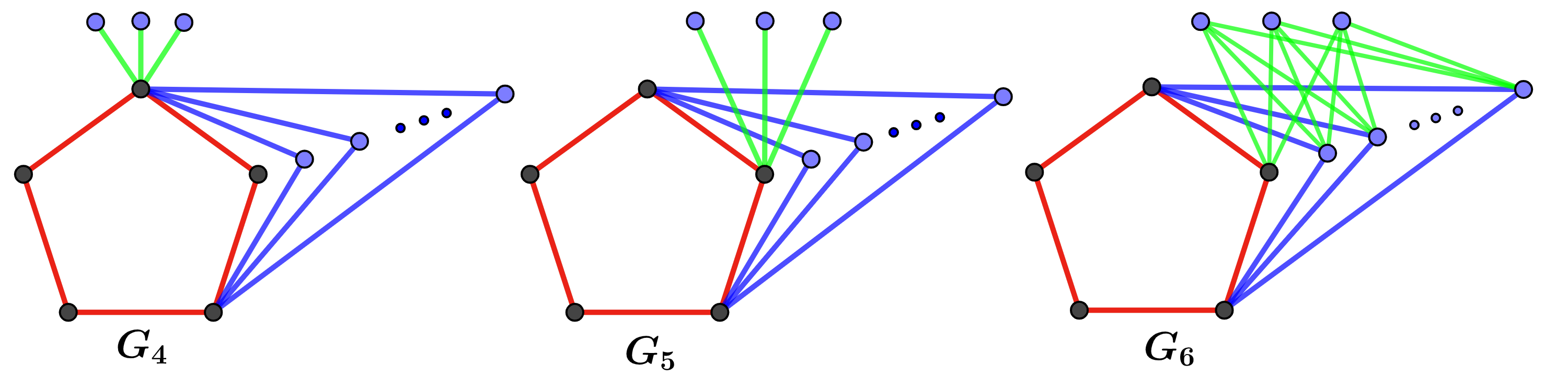}  
\caption{The structure of $G$ when $|V_1|=1$ or $|V_1|=3$.} 
\label{fig-G123} 
\end{figure}

Suppose that all vertices $V_1$ are adjacent to $u_1$. 
Then there is no edge between $V_1$ and $N_{1,3}$ since $G$ 
is triangle-free. Note that $m=5+ 2|N_{1,3}| + |V_1|$ is even, we get $|V_1|\in \{1,3\}$; 
see $L_m$ and $G_4$ in Figure \ref{fig-G123}.

If $|V_1|=1$, then $G$ is the desired extremal graph 
$L_m$;

If $|V_1|=3$, then by computation, we get 
$\lambda (G_4)$ is the largest root of 
\[  F_4(x):=x^6 - mx^4 +(\tfrac{7m}{2} -14)x^2 +(6-m)x + 9-\tfrac{3m}{2}.  \]
Clearly, we can check that  
$L(x)<F_4(x)$ for each $x\ge 1$, and so $\lambda (G_4) < \lambda (L_m)$.

Suppose that all vertices of $V_1$ are adjacent to $u_2$. 
If there is no edge between $V_1$ and $N_{1,3}$, 
then $|V_1|\in \{1,3\}$ and $G$ is isomorphic to $G_2$ or $G_5$; see
Figure \ref{fig-G123}. 
By computations or Lemma \ref{lem-WXH}, we can get $\lambda (G) 
< \lambda (L_m)$;  
If there exists an edge between $V_1$ and $N_{1,3}$, 
then we claim that $V_1$ and $N_{1,3}$ 
form a complete bipartite subgraph by Lemma \ref{lem-L1234}. 
Indeed, Lemma \ref{lem-L1234} asserts that $G$ does not contain 
$L_1$  as an induced subgraph. 
In other words, if $v\in V_1$ is a vertex which is adjacent to one vertex of $N_{1,3}$, 
then $v$ will be adjacent to all vertices of $N_{1,3}$. 
Note that $G$ does not contain $L_2$ as an induced subgraph,  
which means that other vertices of $V_1$ are also adjacent to all vertices of $N_{1,3}$. 
Observe that $m=5+2|N_{1,3}|+ |V_1|(1+|N_{1,3}|)$ is even, 
which yields that $|V_1|$ is odd, and so $|V_1|\in \{1,3\}$. 
Consequently, $G$ is isomorphic to either $Y_m$ or $G_6$; see Figure \ref{fig-G123}.

If $|V_1|=1$, then $|N_{1,3}|= \frac{m-6}{3}$ and $G=Y_m$. 
By Lemma \ref{lem-LY}, we get $\lambda (L_m) < \lambda (Y_m)$. 
Thus,  $Y_m$ is the required extremal graph 
whenever ${m}=3t$ for some even $t\in \mathbb{N}^*$.

If $|V_1|=3$, then $|N_{1,3}|= \frac{m-8}{5}$ and $\lambda (G_6)$
 is the largest root of 
\[  F_6(x):= x^4 - x^3 + (2-m) x^2 + (m-3) x +\tfrac{3m-9}{5}.   \]
It is not hard to check that $\lambda (G_6) < \lambda (L_m)$. 
Indeed, 
by calculation, we know that the largest roots of $x^2F_6(x)$  and 
$L(x)$ are located in $(\sqrt{m-3}$, $\sqrt{m-2})$. Moreover, we denote  
\[ D(x):= L(x)-x^2F_6(x)= x^5 - 2x^4 + (3-m) x^3 + (\tfrac{19m}{10} - \tfrac{26}{5})x^2 + 
(4-m)x + 2 - \tfrac{m}{2}. \]
Clearly, we can verify that $D(\sqrt{m-3}) <0$ 
and $D(\sqrt{m-2}) <0$. 
Furthermore, one can prove that 
$\frac{\mathrm{d}}{\mathrm{d}x} D(x) >0$ 
for each $x\ge \sqrt{m-3}$. 
Consequently, it leads to $D(x)<0$ for every $x\in (\sqrt{m-3}, \sqrt{m-2})$, and so $L(x) < x^2F_6(x) $, 
which yields $\lambda (G_6) < \lambda (L_m)$.

{\bf Subcase 1.2.} 
Without loss of generality, 
we may assume that  $N_{2,4} \neq \varnothing$ and $N_{5,2}= \varnothing$, then 
$N(S)= V_1\cup N_{1,3} \cup N_{2,4}$. 
By Lemma \ref{lem-H123}, $H_2$ can not be an induced subgraph of $G$. 
Thus, $N_{1,3}$ and $N_{2,4}$ induce a complete 
bipartite subgraph in $G$. 
Now, we consider the vertices of $V_1$. 
Recall that all vertices of $V_1$ are adjacent to a same vertex of $S$.  
By Lemma \ref{lem-T1234}, $G$ does not contain $T_3$ and $T_4$ as induced subgraphs. Then the vertices of $V_1$ can not be adjacent to 
$u_1,u_4$ or $u_5$. 
By Lemma \ref{lem-L1234}, we know that $L_3$ and $L_4$ can not be induced subgraph of $G$. Thus, all vertices of $V_1$ can not be adjacent to 
$u_2$ or $u_3$. To sum up, we get $V_1=\varnothing$, 
and so $N(S)=N_{1,3} \cup N_{2,4}$. 
We denote $A=N_{1,3} \cup \{u_2,u_4\}$ 
and $B=N_{2,4}\cup \{u_3,u_1\}$. 
Let $|A|=a $ and $|B|=b$.  
Then we  observe that 
$G$ is isomorphic to the subdivision of 
the complete bipartite graph $K_{a,b}$ by 
subdividing the edge $u_1u_4$ of $K_{a,b}$. 
Note that $m= ab +1$ and $a,b\ge 3$ are odd integers.  
Without loss of generality, we may assume that $a\ge b$. 

If $b=3$, then $m=3a+1$ for some $a\in \mathbb{N}^*$. 
In this case, we get $G=T_m$. 
Invoking Lemma \ref{lem-LT}, we have 
$\lambda (L_m) < \lambda (T_m)$ and thus $T_m$ 
is the desired extremal graph.

If $b\ge 5$, then $m=ab+1$ and
 $\lambda(SK_{a,b})$ is the largest root of 
\[  F_{a,b}(x):=x^5-m x^3+ (3m-2 -2a -2b)x -2m +2a +2b.\]   
Recall in (\ref{eq-Lx}) that $\lambda (L_m)$ is the largest root of 
$ L(x)$.  
We can  verify that 
\[ L(x)- xF_{a,b}(x)= -(\tfrac{m}{2} +5-2a - 2b)x^2 
+ (4+m - 2a - 2b)x -\tfrac{m}{2} +2  .  \]
Since $b\ge 5$ and $m=ab+1$, 
we get $\frac{m}{2} +5-2a-2b = \frac{1}{2}((a-4)(b-4)- 5)>0$. 
It follows that $L(x)\le xF_{a,b}(x)$ for every $x\ge 3$. 
Thus, we get  $\lambda (SK_{a,b}) < \lambda (L_m)$, as required.

\medskip 
 {\bf Case 2.} 
 There is exactly one vertex of $V(G)\setminus S$ with distance $2$ to 
 the cycle $S$. 
 Let $w_2,w_1$ be two vertices with $w_2\sim w_1 \sim S$, 
 and $N_S(w_1)=\{u_1,u_3\}$ by Lemma \ref{lem-J1234}. 
We denote $V(G) \setminus (S\cup \{w_1,w_2\}) := V_1\cup V_2$, 
where $V_i=\{v\in V(G) : v\notin S\cup \{w_1,w_2\}, d_S(v)=i \}$ for 
each $i=1,2$. Similar with the argument in Case 1, 
using Lemmas \ref{lem-T1234} and \ref{lem-WXH}, 
one can move all vertices of $V_1$ such that 
  all of them are adjacent to a same vertex of $S$. 
By Lemma \ref{lem-H123}, $H_1$ is not an induced subgraph of $G$. 
Then $|V_1| \le 3$. 

Let $v\in V_2$ be any vertex. 
We claim that $N_S(v)= \{u_1,u_3\}$. 
Indeed, Lemma \ref{lem-H123} implies that $N_S(v) \neq \{u_3,u_5\}$ 
and $N_S(v)\neq \{u_1,u_4\}$. 
If $N_S(v)=\{u_2,u_4\}$, then $w_1v\in E(G)$  
since $G$ does not contain  $H_2$ as an induced subgraph by Lemma \ref{lem-H123} again.  Consequently, $S\cup \{w_1,w_2,v\}$ forms an induced copy of $L_4$, which contradicts with Lemma \ref{lem-L1234}. 
Thus, we get $N_S(v)\neq \{u_2,u_4\}$. Similarly, we can show 
$N_S(v)\neq \{u_2,u_5\}$. 
In conclusion, we obtain $N_S(v)=\{u_1,u_3\}$ for every $v\in V_2$. 
Since  $m$ is even, we get $|V_1|\in \{0,2\}$.

First of all, suppose that $|V_1|=0$. Then $G$ is isomorphic to $G_2$,  
the graph obtained from a $C_5$ by blowing up the vertex $u_2$ 
exactly $\frac{m-4}{2}$ times and then hanging an edge to $u_2$; 
see Figure \ref{fig-G123}. 
By computations, we
$\lambda (G_2) < \lambda (L_m)$, as desired.

Now, suppose that $|V_1|=2$ 
and $V_1:=\{v_1,v_2\}$. By Lemma \ref{lem-T1234}, we know that $T_3$ and $T_4$ are not induced subgraphs of $G$. Then the vertices of $V_1$ can not be adjacent to 
$u_4$ and $u_5$.  
 By symmetry of $u_1$ and $u_3$, 
there are two possibilities, namely, all vertices of $V_1$ are adjacent to 
$u_1$ or $u_2$. 
If all vertices of $V_1$ are adjacent to $u_1$, then 
$v_1w_2\notin E(G)$ and $v_2w_2\notin E(G)$  
since $T_1$ can not be an induced subgraph of $G$ by Lemma \ref{lem-T1234}. 
By comparing the Perron components of 
$u_1$ and $w_1$, one can move $v_1,v_2$ and $w_2$ together 
using Lemma \ref{lem-WXH}. 
Thus, $G$ is isomorphic to $G_4$ or $G_5$ in Figure \ref{fig-G123}. 
If all vertices of $V_1$ are adjacent to $u_2$, 
then $v_1w_2\notin E(G)$ and $v_2w_2\notin E(G)$  
since $J_3$ is not an induced subgraph in $G$ by Lemma \ref{lem-J1234}. 
A similar argument shows that  $G$ is isomorphic to $G_5$ 
in Figure \ref{fig-G123}. 
By direct computations, we can obtain $\lambda (G_4) < \lambda (L_m) $ 
and $\lambda (G_5) < \lambda (L_m)$. 
This completes the proof. 
\end{proof}

\section{Concluding remarks} 

\label{sec5}
Although we have solved Question \ref{ques-ZS} for every $m\ge 4.7\times 10^5$, 
our proof requires a lot of calculations of eigenvalues.  
As shown in Figure \ref{fig-LFP}, 
there are three kinds extremal graphs depending on 
$m\, (\mathrm{mod}~3) \in \{0,1,2\}$. 
Thus, it seems unavoidable to make calculations  and 
comparisons 
among the spectral radii of these three graphs. 
Unlike the odd case in Theorem \ref{thmZS2022}, 
the bound $\beta (m)$ is sharp for all odd integers $m\in \mathbb{N}$. For the even case, 
Theorem \ref{thm-main} presents all extremal graph for $m\ge 4.7 \times 10^5$. 
We do not try our best to optimize the lower bound on $m$.  
In addition, for $m\in \{6,8,10\}$, 
Lemma \ref{lem-L-m} gives $\lambda (L_{m}) > \sqrt{m-2}$. 
Using a result in \cite[Theorem 5]{Wang2022DM}, 
we can prove that $Y_6, L_8$ and $T_{10}$ are  extremal graphs 
when $m\in \{6,8,10\}$, respectively. 
In view of this evidence, it is possible 
to find a new proof of Question \ref{ques-ZS} to 
characterize the extremal graphs for every $m\ge 12$.

 The blow-up of a graph $G$ is a new graph obtained from $G$ by replacing each vertex $v\in V(G)$ with an independent set $I_v$, 
and for two vertices $u,v\in V(G)$, we add all edges between $I_u$ and $I_v$ 
whenever $uv \in E(G)$.
 It was proved in \cite{LNW2021,Niki2021} that if 
 $G$ is a triangle-free graph with $m\ge 2$ edges, 
 then $\lambda_1^2(G) + \lambda_2^2(G)\le m$, 
 where the equality holds if and only if $G$
  is a blow-up of a member of the family $ \mathcal{G} = \{P_2 \cup K_1, 2P_2 \cup K_1, P_4 \cup K_1, P_5 \cup K_1\}$. 
  This result confirmed the base case of a conjecture of 
  Bollob\'{a}s and Nikiforov \cite{BN2007jctb}. 
  Observe that all  extremal graphs in this result are bipartite graphs. Therefore, it is possible to consider 
  the maximum of $\lambda_1^2(G) + \lambda_2^2(G)$ 
in which $G$ is triangle-free and non-bipartite.

The extremal problem was also studied for 
non-bipartite triangle-free graphs with given number of vertices. 
We write $SK_{s,t}$ for the graph obtained from
 the complete bipartite graph $K_{s,t}$ by subdividing an edge.   
In 2021, Lin, Ning and Wu \cite{LNW2021} proved that 
if $G$ is a non-bipartite triangle-free graph on $n$ vertices, then 
\begin{equation} \label{eq-n-vert}
  \lambda (G) \le 
\lambda \bigl(SK_{\lfloor \frac{n-1}{2}\rfloor,\lceil \frac{n-1}{2}\rceil} 
\bigr),  
\end{equation}
and equality holds if and only if $G=SK_{\lfloor \frac{n-1}{2}\rfloor,\lceil \frac{n-1}{2}\rceil}$. 
Comparing this result with Theorem \ref{thmZS2022},  
one can see that the extremal graphs with given order and size are 
extremely different although both of them are subdivisions of complete bipartite graphs.  
Roughly speaking, the former is nearly balanced, but the latter is exceedingly unbalanced.  

Later, Li and Peng \cite{LP2022second}  extended (\ref{eq-n-vert})  to the non-$r$-partite $K_{r+1}$-free  graphs with $n$ vertices.  
Notice that the extremal graph in (\ref{eq-n-vert})
has many copies of $C_5$. 
 There is another way to extend (\ref{eq-n-vert}) 
 by considering the non-bipartite graphs on $n$ vertices without any copy of 
 $\{C_3,C_5,\ldots ,C_{2k+1}\}$ where $k\ge 2$. 
 This was done 
 by Lin and Guo \cite{LG2021} as well as Li, Sun and Yu
 \cite{LSY2022} independently. 
Subsequently, the corresponding spectral problem for graphs with $m$ edges was  studied in \cite{SL2023, LP2022oddcycle, LLH2022}. 
 However, the extremal graphs in this setting  
 can be achieved only for odd $m$. Hence, we propose the following question for interested readers\footnote{We believe intuitively that the spectral extremal graphs with even size are perhaps constructed from 
 those in Figure \ref{fig-LFP} by `replacing' the red copy of $C_5$ with a longer odd cycle $C_{2k+3}$.}.

\begin{question}
For even $m$, what is the extremal graph 
attaining the maximum spectral radius over all 
non-bipartite $\{C_3,C_5,\ldots ,C_{2k+1}\}$-free graphs with $m$ edges?
\end{question}

We write $q(G)$ for the 
  signless Laplacian spectral radius, i.e., 
 the largest eigenvalue of 
 the {\it signless Laplacian matrix}  $Q(G)=D(G) + 
 A(G)$, where $D(G)= \mathrm{diag} (d_1,\ldots ,d_n)$ 
 is the degree diagonal matrix and 
 $A(G)$ is the adjacency matrix.   
A theorem of He, Jin and Zhang \cite{HJZ2013} implies that if $G$  is a triangle-free graph on $n$ vertices, 
then $q(G) \le n$, with equality if and only if $G$ is a complete bipartite graph (need not be balanced). This result can also be viewed as a spectral version of Mantel's theorem. 
It is worth mentioning that  Liu, Miao and Xue \cite{LMX2022} 
characterized the maximum signless Laplacian spectral radius among 
all non-bipartite triangle-free graphs with given order $n$ and size $m$, respectively. Fortunately, the corresponding extremal graphs 
are independent of the parity of $m$. 
Soon after, they \cite{MLX2022} also provided the extensions for graphs 
without any copy of $\{C_3,C_5,\ldots ,C_{2k+1}\}$.


\subsection*{Acknowledgements}
This work was supported by  NSFC 
(Grant No. 11931002 and 12271527) 
and Natural Science Foundation of Hunan
Province (Grant No. 2020JJ4675 and 2021JJ40707). 
We thank the referees for their careful review and valuable suggestions 
which improved the presentation.

\appendix

\section{Proofs of Lemmas \ref{lem-H123}, \ref{lem-T1234}, 
\ref{lem-J1234} and \ref{lem-L1234}}

In the appendix, we shall provide the detailed proof of some lemmas 
in Section \ref{sec3}. 
 
 \begin{proof}[{\bf Proof of Lemma \ref{lem-H123}}]
 Suppose on the contrary that $G$ contains $H_i$ 
 as an induced subgraph for some $i\in \{1,2,3\}$. 
 To obtain a contradiction, we shall show that $t(G)>0$ 
 by using Lemma \ref{lem21}. 
  The eigenvalues of graphs $H_1,H_2$ and $H_3$ can be seen in Table \ref{tab-H123}. 
 
 \begin{table}[H]
\centering 
\begin{tabular}{cccccccccc}
\toprule
    & $\lambda_1$  & $\lambda_2$  &  $\lambda_3$ 
    &  $\lambda_4$ &  $\lambda_5$ &  $\lambda_6$ & $\lambda_7$  & $\lambda_8$ & $\lambda_9$ \\ 
\midrule 
$H_1$ & 2.578  &  1.373 & 0.618 & 0 & 0& 0&  $-0.451$ & $-1.618$ & $-2.501$
 \\ 
$H_2$ & 2.641  &  1 &  0.723 & 0.414 & $-0.589$ & $-1.775$ & $-2.414$  \\ 
$H_3$ & 2.681 & 1 & 0.642 & 0 & 0 & $-2$ & $-2.323$   \\ 
\bottomrule 
\end{tabular} 
\caption{Eigenvalues of $H_1,H_2$ and $H_3$.} 
\label{tab-H123} 
\end{table}
 
 First of all, 
 we consider the case that $H_1$ is an induced subgraph 
 in $G$. 
 The Cauchy interlacing theorem implies 
 $\lambda_{n-9+i}(G) \le \lambda_i(H_1) \le \lambda_i (G)
 $ for every $i\in \{1,2,\ldots ,9\}$. 
 We denote $\lambda_i=\lambda_i(G)$ for short. 
Obviously, we have 
 \[  f(\lambda_2 ) \ge f(1.371)\ge 1.879\sqrt{m-2.5} +2.576 \]
 and 
 \[ f(\lambda_3 ) \ge f(0.618)\ge 0.381\sqrt{m-2.5} + 0.236.  \]
 Moreover, for each $i\in \{4,5,6\}$, we know that 
 $\lambda_i\ge 0$, 
 which gives $f(\lambda_i) \ge 0$. 
Next, we shall consider the negative eigenvalues of $G$. 
 The Cauchy interlacing theorem implies 
 $\lambda_{n-2} \le \lambda_7(H_1)=-0.451$ 
 and $\lambda_{n-1} \le \lambda_8(H_1)=-1.618$ 
 and $\lambda_n \le \lambda_9 (H_1)=-2.501$. 
Moreover, we get from Lemma \ref{lem-L-m} that 
$\lambda_1 \ge \lambda (L_m) > \sqrt{m-2.5}$ and 
$\lambda_n^2  
 \le 2m- (\lambda_1^2 + \lambda_2^2 + \lambda_3^2 + \lambda_{n-2}^2 + \lambda_{n-1}^2 ) 
 \le 2m-(m-2.5+5.091)=m-2.591$, which implies $-\sqrt{m-2.591} < \lambda_n \le 
 -2.501$. By Lemma \ref{lem-fx}, we have 
 \[  f(\lambda_n ) \ge 
 \min\{ f(-\sqrt{m-2.591}), f(-2.501)\} > 0.04 \sqrt{m-2.5}. \]
Since $\lambda_{n-1}^2 + \lambda_n^2  
\le 2m- (\lambda_1^2 +\lambda_2^2 + 
\lambda_3^2 + \lambda_{n-2}^2 ) 
< m+0.026$ and $\lambda_{n-1}^2 \le \lambda_n^2$, 
we get $-\sqrt{(m+0.026)/2}< \lambda_{n-1} \le -1.618$. 
 By Lemma \ref{lem-fx},  we get 
 \[  f(\lambda_{n-1}) \ge 
 \min\{ f(-\sqrt{(m+0.026)/{2}}), f(-1.618)\} > 2.617 \sqrt{m-2.5} - 4.235. \] 
 Moreover, we have $-\sqrt{{(m+0.23)}/{3}}<\lambda_{n-2}\le -0.451$ and then 
 \[ f(\lambda_{n-2}) \ge \min \{f(-\sqrt{({m+0.23})/{3}}), 
 f(-0.451)\} >0.203\sqrt{m-2.5} -0.091.  \]
 Theorem \ref{thmZS2022} and  Lemma \ref{lem-beta-m} imply 
\[   \lambda_1<  \beta (m)<\sqrt{m-1.85}  .\] 
 By Lemma \ref{lem21}, we obtain  
\begin{align*} 
t(G)  &> \frac{1}{6}(f(\lambda_2) + f(\lambda_3) + f(\lambda_{n-2})
+f(\lambda_{n-1}) + f(\lambda_n )) - \frac{2.5}{3} \lambda_1 \\ 
&>\frac{1}{6}(5.12\sqrt{m-2.5}-5\sqrt{m-1.85} -1.514)>0,
\end{align*}
where the last inequality holds for $m\ge 188$. This is a contradiction. 

Second, 
assume that $H_2$ is an induced subgraph of $G$.   
 Then Cauchy interlacing theorem gives 
$\lambda_2 \ge 1, \lambda_3\ge 0.723$ and 
$\lambda_4 \ge 0.414$. Similarly, we get 
\[ \begin{aligned} 
f(\lambda_2) &\ge f(1)= \sqrt{m-2.5} +1, \\ 
 f(\lambda_3) &\ge f(0.723) \ge 0.522 \sqrt{m-2.5} +0.377  
\end{aligned} \] 
and 
\[ f(\lambda_4) \ge f(0.414)\ge 0.171\sqrt{m-2.5} + 0.07.  \]
The negative eigenvalues of $H_2$ imply that 
$\lambda_{n-2}\le -0.589$, $\lambda_{n-1}\le -1.775$ 
and $\lambda_n \le -2.414$. 
As $\lambda_n^2 \le 2m- (\lambda_1^2+\lambda_2^2 + \lambda_3^2 + 
\lambda_4^2 + \lambda_{n-2}^2 +\lambda_{n-1}^2) 
< 2m - (m-2.5 + 5.191)=m-2.691$, 
 we  get  $-\sqrt{m-2.691} \le \lambda_n \le -2.414$.   
 Lemma \ref{lem-fx} gives 
\[  f(\lambda_n) \ge \min\{ f(-\sqrt{m-2.691}), f(-2.414)\} 
> 0.09\sqrt{m-2.5}. \] 
In addition, 
we have $-\sqrt{({m+0.459})/{2}} \le \lambda_{n-1} \le -1.775$ and 
\[ f(\lambda_{n-1}) \ge \min\{ f(-\sqrt{{(m+ 0.459)}/{2}}), f(-1.775)\} 
> 3.15 \sqrt{m-2.5} - 5.592 .  \]
 Moreover, we get 
$ \sqrt{(m+0.805)/3}<\lambda_{n-2} \le -0.589$ and 
\[  f(\lambda_{n-2}) \ge 
\min\{f(-\sqrt{(m+0.805)/3}), f(-0.589)\} > 0.346\sqrt{m-2.5} - 0.204.  \]  
Using  Lemma \ref{lem21} and $\lambda_1 < \beta (m)<\sqrt{m-1.85}$, 
 we obtain 
\begin{align*} 
t(G)  &> \frac{1}{6}(f(\lambda_2) + f(\lambda_3) + f(\lambda_4) + 
f(\lambda_{n-2}) + f(\lambda_{n-1}) + f(\lambda_n )) - \frac{2.5}{3} \lambda_1 \\
&> \frac{1}{6}(5.279\sqrt{m-2.5}-5\sqrt{m-1.85} -4.349)>0, 
\end{align*}
where the last inequality holds for $m\ge 258$, 
which is also a contradiction. 

Finally, 
if $H_3$ is an induced subgraph of $G$, 
 then we get $\lambda_2\ge 2$ and $\lambda_3\ge 0.642$. 
 Thus  
\[  f(\lambda_2) \ge f(1)=\sqrt{m-2.5} +1 \]
and 
\[  f(\lambda_3) \ge f(0.642) \ge 0.412\sqrt{m-2.5} +0.264. \] 
Moreover, Cauchy interlacing theorem gives $\lambda_{n-1}\le -2$ 
and $\lambda_n \le -2.323$. 
Since $\lambda_n^2\le 2m-(\lambda_1^2+\lambda_2^2+\lambda_3^2 + \lambda_{n-1}^2)< 2m-(m-2.5 +5.412)=m-2.912$, we get 
$-\sqrt{m-2.912} < \lambda_n \le -2.323$. Then 
\[  f(\lambda_n) \ge \min\{ f(-\sqrt{m-2.912}), f(-2.323)\} \ge 
0.2\sqrt{m-2.5}.  \]
Similarly,  
we have $-\sqrt{(m+1.087)/2} < \lambda_{n-1}\le -2$ and 
\[  f(\lambda_{n-1}) \ge \min\{ f(-\sqrt{(m+1.087)/2}), f(-2)\} 
\ge 4\sqrt{m-2.5} -8. \]
Combining Lemma \ref{lem21} with $\lambda_1 < \sqrt{m-1.85}$, 
we get 
\begin{align*} 
t(G)  &> \frac{1}{6}(f(\lambda_2) + f(\lambda_3) + 
f(\lambda_{n-1}) + f(\lambda_n )) - \frac{2}{3} \lambda_1 \\
&> \frac{1}{6}( 5.612\sqrt{m-2.5}-5\sqrt{m-1.85} -6.736 )>0, 
\end{align*}
where the last inequality holds for $m\ge 162$, which is a contradiction. 
 \end{proof}

Using the similar method as in the proofs of Lemmas \ref{lem-C5} and \ref{lem-H123}, 
we can prove Lemmas \ref{lem-T1234}, \ref{lem-J1234} and \ref{lem-L1234} as well. 
For simplicity, we next present a brief sketch only.  

\begin{proof}[{\bf Proof of Lemma \ref{lem-T1234}}] 
First of all, the eigenvalues of $T_1,\ldots ,T_4$ can be given as below. 

\begin{table}[H]
\centering 
\begin{tabular}{cccccccc}
\toprule
    & $\lambda_1$  & $\lambda_2$  &  $\lambda_3$ 
    &  $\lambda_4$ &  $\lambda_5$ &  $\lambda_6$ & $\lambda_7$   \\ 
\midrule
 $T_1$ & 2.377 & 1.273 & 0.801 & 0 & $-0.554$ & $-1.651$ &
 $-2.246$ \\ 
$T_2$ & 2.342  &  1 & 1 & 0.470 & $-1$ &  $-1.813$ 
& $-2$  \\ 
$T_3$ & 2.641  &  1 &  0.723 & 0.414 & $-0.589$ & 
$-1.775$ & $-2.414$   \\ 
$T_4$ & 2.447 & 1.176 & 0.656 & 0 & $-0.264$ & $-1.832$ & $-2.183$    \\ 
\bottomrule 
\end{tabular} 
\caption{Eigenvalues of $T_1,T_2,T_3$ and $T_4$.} 
\label{tab-T1234} 
\end{table}

Suppose on the contrary that $T_1$ is an induced subgraph of $G$. 
Then Lemma \ref{lemCauchy} gives 
$\lambda_2 \ge 1.273$ and $\lambda_3 \ge 0.801$. 
Thus, we get  
\[  f(\lambda_2) \ge f(1.273) \ge 1.62\sqrt{m-2.5} + 2.062  \]
and 
\[ f(\lambda_3) \ge f(0.801) \ge 0.641 \sqrt{m-2.5} + 0.513.  \]
Moreover, using the same technique in Lemma \ref{lem-H123}, 
the negative eigenvalues of $T_1$ implies 
$\lambda_n^2 \le 2m- (\lambda_1^2 + \lambda_2^2+ \lambda_3^2 
+ \lambda_{n-2}^2+ \lambda_{n-1}^2) < 
2m- (m-2.5 + 5.299)=m-2.799$, and so 
$-\sqrt{m-2.799}<\lambda_n \le -2.246$. By Lemma \ref{lem-fx}, 
it follows that 
\[  f(\lambda_n) \ge \min\{f(-\sqrt{m-2.799}), 
f(-2.246)\} > 0.14 \sqrt{m-2.5}. \]
Similarly, we can get 
\[  f(\lambda_{n-1}) \ge f(-1.651)\ge 
2.725 \sqrt{m-2.5} - 4.5 \] 
and 
\[ f(\lambda_{n-2}) \ge f(-0.554) \ge 0.306 \sqrt{m-2.5} - 0.17.  \] 
Using  Lemma \ref{lem21} and $\lambda_1 < \beta (m)<\sqrt{m-1.85}$, 
 we obtain   
\begin{align*} 
t(G)  &> \frac{1}{6}(f(\lambda_2) + f(\lambda_3) +
f(\lambda_{n-2}) + f(\lambda_{n-1}) + f(\lambda_n )) - \frac{2.5}{3} \lambda_1 \\
&> \frac{1}{6}(5.432\sqrt{m-2.5}-5\sqrt{m-1.85} -2.095)>0, 
\end{align*}
which is a contradiction. Thus, $G$ does not contain $T_1$ 
as an induced subgraph. 

Suppose on the contrary that $G$ contains $T_2$ as an induced subgraph. 
Then Lemma \ref{lemCauchy} implies 
$\lambda_2\ge 1, \lambda_3\ge 1 $ and $\lambda_4\ge 0.47$. 
Thus, we have 
\[ \begin{aligned}
f(\lambda_2)&\ge f(1)= \sqrt{m-2.5} +1,  \\
 f(\lambda_3) &\ge f(1)= \sqrt{m-2.5} +1  
\end{aligned} \] 
and 
\[  f(\lambda_4) \ge f(0.47) \ge 0.22\sqrt{m-2.5} + 0.103. \]
Moreover, we have   
$ -\sqrt{m-4.01} < \lambda_n \le -2$. Then Lemma \ref{lem-fx} leads to 
\[  f(\lambda_n) \ge \min\{ f(-\sqrt{m-4.01} ), f(-2)\} \ge 
0.7\sqrt{m-2.5}.  \]
In addition, we have 
\[  f(\lambda_{n-1}) \ge f(-1.813) \ge 
3.28 \sqrt{m-2.5} - 5.959 \]
and 
\[  f(\lambda_{n-2}) \ge f(-1) = \sqrt{m-2.5} -1. \]
By  Lemma \ref{lem21} and $\lambda_1 < \beta (m)<\sqrt{m-1.85}$, 
 it follows that  
\begin{align*} 
t(G)  &> \frac{1}{6}(f(\lambda_2) + f(\lambda_3) + f(\lambda_4) + 
f(\lambda_{n-2}) + f(\lambda_{n-1}) + f(\lambda_n )) - \frac{2.5}{3} \lambda_1 \\
&> \frac{1}{6}(7.2\sqrt{m-2.5}-5\sqrt{m-1.85} -4.856)>0, 
\end{align*}
 a contradiction. So $G$ does not contain $T_2$ 
as an induced subgraph.

Suppose on the contrary that $T_3$ is an induced subgraph of $G$. 
Using Lemma \ref{lemCauchy}, we obtain 
$\lambda_2\ge 1, \lambda_3\ge 0.723$ and $\lambda_4 \ge 0.414$. 
Then 
\begin{equation*}
\begin{aligned} 
 f(\lambda_2) &\ge f(1)=\sqrt{m-2.5} +1, \\
 f(\lambda_3) &\ge f(0.723) \ge 0.522 \sqrt{m-2.5} + 0.377 
 \end{aligned}
 \end{equation*}
 and 
 \[  f(\lambda_4) \ge f(0.414) \ge 0.171\sqrt{m-2.5} +0.07. \]
 Moreover, we can get 
 $-\sqrt{m-2.695} \le \lambda_n \le -2.414$. By Lemma \ref{lem-fx}, 
 we have 
 \[  f(\lambda_n) \ge \min\{f(-\sqrt{m-2.695} ), f(-2.414)\} 
 \ge 0.08 \sqrt{m-2.5}.  \]
Similarly, we obtain 
\[  f(\lambda_{n-1}) \ge f(-1.775) \ge 3.15 \sqrt{m-2.5} - 5.592  \]
and 
\[  f(\lambda_{n-2}) \ge f(-0.589) \ge 0.346 \sqrt{m-2.5} - 0.204. \]
Consequently, Lemma \ref{lem21} and $\lambda_1 < \beta (m)<\sqrt{m-1.85}$ 
gives 
\begin{align*} 
t(G)  &> \frac{1}{6}(f(\lambda_2) + f(\lambda_3) + f(\lambda_4) + 
f(\lambda_{n-2}) + f(\lambda_{n-1}) + f(\lambda_n )) - \frac{2.5}{3} \lambda_1 \\
&> \frac{1}{6}(5.269\sqrt{m-2.5}-5\sqrt{m-1.85} -4.349)>0, 
\end{align*}
where the last inequality holds for $m\ge 276$. 
This leads to a contradiction. Hence  $T_3$ 
can not be an induced subgraph of $G$.  

Suppose on the contrary that $T_4$ is an induced subgraph of $G$. 
Applying Lemma \ref{lemCauchy}, we obtain 
$\lambda_2 \ge 1.176$ and $\lambda_3 \ge 0.656$. Thus 
\[  f(\lambda_2) \ge f(1.176) \ge 1.382 \sqrt{m-2.5} +1.626 \]
and 
\[  f(\lambda_3) \ge f(0.656) \ge 0.43 \sqrt{m-2.5} + 0.282.  \]
Moreover, we have $-\sqrt{m-2.741} \le \lambda_n \le -2.183$. 
Lemma \ref{lem-fx} implies 
\[  f(\lambda_n ) \ge \min \{f(-\sqrt{m-2.741}), 
f(-2.183)\} \ge 0.1 \sqrt{m-2.5}. \]
Similarly, one can get 
\[  f(\lambda_{n-1}) \ge f(-1.832) \ge 3.356\sqrt{m-2.5} - 6.148 \] 
and 
\[ f(\lambda_{n-2}) \ge f(-0.264) \ge 0.069 \sqrt{m-2.5} - 0.018. \]
Finally, combining Lemma \ref{lem21} with $\lambda_1 < \beta (m)<\sqrt{m-1.85}$, we obtain 
\begin{align*} 
t(G)  &> \frac{1}{6}(f(\lambda_2) + f(\lambda_3) + 
f(\lambda_{n-2}) + f(\lambda_{n-1}) + f(\lambda_n )) - \frac{2.5}{3} \lambda_1 \\
&> \frac{1}{6}(5.337\sqrt{m-2.5}-5\sqrt{m-1.85} -4.258)>0, 
\end{align*}
 a contradiction. Therefore  $T_4$ 
can not be an induced subgraph of $G$.  
\end{proof}

The proofs of Lemmas \ref{lem-J1234} and \ref{lem-L1234} can 
proceed in a similar way.

\begin{proof}[{\bf Proof of Lemma \ref{lem-J1234}}]
The eigenvalues of graphs $J_1,\ldots ,J_4$ can be computed
 as follows. 
 
\begin{table}[H]
\centering 
\begin{tabular}{cccccccccc}
\toprule
    & $\lambda_1$  & $\lambda_2$  &  $\lambda_3$ 
    &  $\lambda_4$ &  $\lambda_5$ &  $\lambda_6$ & $\lambda_7$  & $\lambda_8$ & $\lambda_9$ \\ 
\midrule 
$J_1$ & 2.151 & 1.268 & 0.618 & 0.420 & $-0.895$ & $-1.618$ &
 $-1.944$ \\  
$J_2$ & 2.554  &  1.223 & 0.618 & 0.565 & 0&  $-0.942$ & $-1.618$ & $-2.401$
 \\ 
$J_3$ & 2.900  &  1.362 &  0.690 & 0.618 & 0.618 & $-0.273$ & $-1.618$ & $-1.618$ & $-2.679$  \\ 
$J_4$ & 3.082 & 1.380 & 0.827 & 0.670 & 0.338 & $-0.406$ & $-1.209$ & $-1.726$ & $-2.956$   \\ 
\bottomrule 
\end{tabular} 
\caption{Eigenvalues of $J_1,J_2,J_3$ and $J_4$.} 
\label{tab-J1234} 
\end{table}

If $J_1$ is an induced subgraph of $G$, 
then Lemma \ref{lemCauchy} implies 
$\lambda_2 \ge 1.268, \lambda_3 \ge 0.618$ 
and $\lambda_4\ge 0.42$. 
Moreover, the negative eigenvalues of $J_1$ 
gives $\lambda_{n-2} \le -0.895$ and 
 $\lambda_{n-1}\le -1.618$. 
Then  $-\sqrt{m-3.085} \le \lambda_n \le -1.944$ and 
\[ f(\lambda_n) \ge
\min\{f(-\sqrt{m-3.085}), f(-1.944)\} > 0.25\sqrt{m-2.5}.  \] 
By Lemma \ref{lem21} and $\lambda_1 < \beta (m)<\sqrt{m-1.85}$, we obtain 
\begin{align*} 
t(G)  &> \frac{1}{6}(f(\lambda_2) + f(\lambda_3) + f(\lambda_4) + 
f(\lambda_{n-2}) + f(\lambda_{n-1}) + f(\lambda_n )) - \frac{2.5}{3} \lambda_1 \\
&> \frac{1}{6}(5.835\sqrt{m-2.5}-5\sqrt{m-1.85} - 2.603)>0, 
\end{align*}
 a contradiction. Thus $J_1$ 
can not be an induced subgraph of $G$.  

If $J_2$ is an induced subgraph of $G$, 
then Lemma \ref{lemCauchy} implies 
$\lambda_2 \ge 1.223, \lambda_3 \ge 0.618$ 
and $\lambda_4\ge 0.565$. 
In addition, the negative eigenvalues of $J_2$ 
gives $\lambda_{n-2} \le -0.942$ and 
 $\lambda_{n-1}\le -1.618$. 
Then  $-\sqrt{m-3.202} \le \lambda_n \le -2.401$ and 
\[ f(\lambda_n) \ge
\min\{f(-\sqrt{m-3.202}), f(-2.401)\} > 0.3\sqrt{m-2.5}.  \] 
Applying Lemma \ref{lem21} and $\lambda_1 < \beta (m)<\sqrt{m-1.85}$, we have  
\begin{align*} 
t(G)  &> \frac{1}{6}(f(\lambda_2) + f(\lambda_3) + f(\lambda_4) + 
f(\lambda_{n-2}) + f(\lambda_{n-1}) + f(\lambda_n )) - \frac{2.5}{3} \lambda_1 \\
&> \frac{1}{6}(6.002\sqrt{m-2.5}-5\sqrt{m-1.85} - 2.826)>0, 
\end{align*}
 a contradiction. So $J_2$ 
can not be an induced subgraph of $G$. 

If $J_3$ is an induced subgraph of $G$, 
then Lemma \ref{lemCauchy} implies 
$\lambda_2 \ge 1.362, \lambda_3 \ge 0.690$, 
 $\lambda_4\ge 0.618$ and $\lambda_5\ge 0.618$. 
Additionally, the negative eigenvalues of $J_3$ 
gives $\lambda_{n-3} \le -0.273$, 
$\lambda_{n-2} \le -1.618$ and 
 $\lambda_{n-1}\le -1.618$. 
Then  $-\sqrt{m-5.907} \le \lambda_n \le -2.679$ and 
\[ f(\lambda_n) \ge
\min\{f(-\sqrt{m-5.907}), f(-2.679)\} > 1.5\sqrt{m-2.5}.  \] 
Using Lemma \ref{lem21} and $\lambda_1 < \beta (m)<\sqrt{m-1.85}$, we get 
\begin{align*} 
t(G)  &> \frac{1}{6}(f(\lambda_2) +\cdots +  
f(\lambda_5)+ f(\lambda_{n-3}) + \cdots + f(\lambda_n )) - \frac{2.5}{3} \lambda_1 \\
&> \frac{1}{6}(9.907\sqrt{m-2.5}-5\sqrt{m-1.85} - 5.164)>0, 
\end{align*}
 a contradiction. Therefore $J_3$ 
can not be an induced subgraph of $G$.

If $J_4$ is an induced subgraph of $G$, 
then Lemma \ref{lemCauchy} implies 
$\lambda_2 \ge 1.380, \lambda_3 \ge 0.827$, 
 $\lambda_4\ge 0.670$ and $\lambda_5\ge 0.338$. 
Furthermore, the negative eigenvalues of $J_4$ 
gives $\lambda_{n-3} \le -0.406$, 
$\lambda_{n-2} \le -1.209$ and 
 $\lambda_{n-1}\le -1.726$. 
Then  $-\sqrt{m-5.259} \le \lambda_n \le -2.956$ and 
\[ f(\lambda_n) \ge
\min\{f(-\sqrt{m-5.259}), f(-2.956)\} > 1.3\sqrt{m-2.5}.  \] 
Due to Lemma \ref{lem21} and $\lambda_1 < \beta (m)<\sqrt{m-1.85}$, we obtain 
\begin{align*} 
t(G)  &> \frac{1}{6}(f(\lambda_2) + \cdots + f(\lambda_5) + 
f(\lambda_{n-3}) +\cdots + f(\lambda_n )) - \frac{2.5}{3} \lambda_1 \\
&> \frac{1}{6}(9.059\sqrt{m-2.5}-5\sqrt{m-1.85} - 3.442)>0, 
\end{align*}
 a contradiction. Henceforth $J_4$ 
can not be an induced subgraph of $G$. 
\end{proof}

\begin{proof}[{\bf Proof of Lemma \ref{lem-L1234}}] 
By computation, we can obtain the eigenvalues of $L_1, \ldots ,L_4$. 

\begin{table}[H]
\centering 
\begin{tabular}{ccccccccc}
\toprule
    & $\lambda_1$  & $\lambda_2$  &  $\lambda_3$ 
    &  $\lambda_4$ &  $\lambda_5$ &  $\lambda_6$ & $\lambda_7$ & $\lambda_8$   \\ 
\midrule
 $L_1$ & 2.950 & 1.156 & 0.618 & 0.522 & 0 & $-0.790$ &
 $-1.618$ & $-2.838$ \\ 
$L_2$ & 2.753  &  1.204 & 0.641 & 0.618 & $-0.253$ &  $-0.700$ 
& $-1.618$ &$-2.645$  \\ 
$L_3$ & 3.141 & 1.139 & 0.763 & 0 & 0 & $-0.277$ & $-1.745$ & 
$-3.021$    \\  
$L_4$ & 2.964 & 1 & 0.764 & 0.513 & 0 & $-0.710$ & $-1.722$ & 
$-2.809$    \\ 
\bottomrule 
\end{tabular} 
\caption{Eigenvalues of $L_1,L_2,L_3$ and $L_4$.} 
\label{tab-L1234} 
\end{table}

Suppose on the contrary that $G$ contains 
$L_1$ as an induced subgraph. 
Then Lemma \ref{lemCauchy} yields 
$\lambda_2\ge 1.156$, $\lambda_3\ge 0.618$ and 
$\lambda_4\ge 0.522$. The negative eigenvalues of $L_1$ 
implies $\lambda_{n-2} \le -0.790$ and 
$\lambda_{n-1} \le -1.618$. Then $- \sqrt{m-2.734} \le 
\lambda_n \le -2.838$. By Lemma \ref{lem-fx}, we have 
\[ f(\lambda_n) \ge \min\{f(- \sqrt{m-2.734}), f(-2.838) \}  > 
0.1 \sqrt{m-2.5}. \] 
According to 
Lemma \ref{lem21} and $\lambda_1 < \beta (m)<\sqrt{m-1.85}$,  
it follows that 
\[ \begin{aligned} 
t(G)  &> \frac{1}{6}(f(\lambda_2) + f(\lambda_3) + f(\lambda_4) + 
f(\lambda_{n-2}) + f(\lambda_{n-1}) + f(\lambda_n )) - \frac{2.5}{3} \lambda_1 \\
&> \frac{1}{6}(5.334\sqrt{m-2.5}-5\sqrt{m-1.85} - 2.805)>0, 
\end{aligned} \] 
which is a contradiction. 
Thus $G$ does not contain $L_1$ as an induced subgraph.

Suppose on the contrary that $G$ contains 
$L_2$ as an induced subgraph. 
 Lemma \ref{lemCauchy} yields 
$\lambda_2\ge 1.204$, $\lambda_3\ge 0.641$ and 
$\lambda_4\ge 0.618$. 
The negative eigenvalues of $L_2$ 
implies $\lambda_{n-3} \le -0.253$, 
$\lambda_{n-2} \le -0.7$ and 
$\lambda_{n-1} \le -1.618$. Then $- \sqrt{m-2.918} \le 
\lambda_n \le -2.645$. Lemma \ref{lem-fx} gives  
\[ f(\lambda_n) \ge \min\{f(- \sqrt{m-2.918}), f(-2.645) \}  > 
0.2 \sqrt{m-2.5}. \] 
By Lemma \ref{lem21} and $\lambda_1 < \beta (m)<\sqrt{m-1.85}$,  
it follows that 
\[ \begin{aligned} 
t(G)  &> \frac{1}{6}(f(\lambda_2) + f(\lambda_3) + f(\lambda_4) + 
f(\lambda_{n-3}) + \cdots + f(\lambda_n )) - \frac{2.5}{3} \lambda_1 \\
&> \frac{1}{6}(5.618\sqrt{m-2.5}-5\sqrt{m-1.85} - 2.35)>0, 
\end{aligned} \] 
which is a contradiction. 
Therefore $G$ does not contain $L_2$ as an induced subgraph.

Suppose on the contrary that $G$ contains 
$L_3$ as an induced subgraph. 
 Lemma \ref{lemCauchy} yields 
$\lambda_2\ge 1.139$ and $\lambda_3\ge 0.763$. 
The negative eigenvalues of $L_3$ 
implies 
$\lambda_{n-2} \le -0.277$ and 
$\lambda_{n-1} \le -1.745$. Then $- \sqrt{m-2.504} \le 
\lambda_n \le -3.021$. Lemma \ref{lem-fx} gives  
\[ f(\lambda_n) \ge \min\{f(- \sqrt{m-2.504}), f(-3.021) \}  > 
0.001 \sqrt{m-2.5}. \] 
By Lemma \ref{lem21} and $\lambda_1 < \beta (m)<\sqrt{m-1.85}$,  
it follows that for $m\ge 4.7 \times 10^5$, 
\[ \begin{aligned} 
t(G)  &> \frac{1}{6}(f(\lambda_2) + f(\lambda_3) + f(\lambda_{n-2}) + 
f(\lambda_{n-1})  + f(\lambda_n )) - \frac{2.5}{3} \lambda_1 \\
&> \frac{1}{6}(5.005\sqrt{m-2.5}-5\sqrt{m-1.85} - 3.412)>0, 
\end{aligned} \] 
which is a contradiction. 
Thus, $G$ does not contain $L_3$ as an induced subgraph. 

Suppose on the contrary that $G$ contains 
$L_4$ as an induced subgraph. 
 Lemma \ref{lemCauchy} yields 
$\lambda_2\ge 1$, $\lambda_3\ge 0.764$ 
and $\lambda_4 \ge 0.513$. 
The negative eigenvalues of $L_4$ 
implies 
$\lambda_{n-2} \le -0.710$ and 
$\lambda_{n-1} \le -1.722$. Then $- \sqrt{m-2.818} \le 
\lambda_n \le -2.809$. Lemma \ref{lem-fx} gives  
\[ f(\lambda_n) \ge \min\{f(- \sqrt{m-2.818}), f(-2.809) \}  > 
0.15 \sqrt{m-2.5}. \] 
By Lemma \ref{lem21} and $\lambda_1 < \beta (m)<\sqrt{m-1.85}$,  
it follows that 
\[ \begin{aligned} 
t(G)  &> \frac{1}{6}(f(\lambda_2) + f(\lambda_3) + f(\lambda_4) 
 + f(\lambda_{n-2}) + 
f(\lambda_{n-1})  + f(\lambda_n )) - \frac{2.5}{3} \lambda_1 \\
&> \frac{1}{6}(5.468\sqrt{m-2.5}-5\sqrt{m-1.85} - 3.883)>0, 
\end{aligned} \] 
which is a contradiction. 
Hence $G$ does not contain $L_4$ as an induced subgraph.
\end{proof}

\end{document}